\documentclass[english,10pt]{amsart}

\usepackage{amsmath, amssymb, amsthm, enumerate}
\usepackage[all]{xy}
\usepackage[activeacute]{babel}
\usepackage[alphabetic]{amsrefs}

\newtheorem{thm}[equation]{Theorem}
\makeatletter
\let\c@subsubsection\c@equation
\makeatother
\newtheorem{lem}[equation]{Lemma}
\newtheorem{prop}[equation]{Proposition}

\theoremstyle{remark}
\newtheorem{rmk}[equation]{Remark}

\theoremstyle{definition}
\newtheorem{defi}[equation]{Definition}

\newcommand{\Hom}{\mathrm{Hom}}
\newcommand{\inthomeff}{\mathbf{hom}^{\mathrm{eff}}}
\newcommand{\Gm}{\mathbb G _{m}}

\newcommand{\spec}[1]{\mathrm{Spec}(#1)}
\newcommand{\generators}{\mathcal G}

\newcommand{\sphere}{\mathbf 1}
\DeclareMathOperator*{\hocolim}{hocolim}
\newcommand{\colim}{\mathrm{colim}}

\newcommand{\DMeff}[1]{DM^{\mathrm{eff}}_{#1}}
\newcommand{\DM}[1]{DM_{#1}}
\newcommand{\northogonal}[2]{DM _{#1}^{\perp}(#2)}
\newcommand{\HINST}[1]{\mathbf{HI}_{#1}}

\newcommand{\hr}{\sphere _{R}}

\numberwithin{equation}{subsection}

\begin{document}


\title{Incidence equivalence and the Bloch-Beilinson filtration}



\author{Pablo Pelaez}
\address{Chebyshev Laboratory\\
St. Petersburg State University\\
14th Line V. O., 29B\\
Saint Petersburg 199178 Russia}
\email{pablo.pelaez@gmail.com}

\author{Araceli Reyes}
\address{Instituto de Matem\'aticas, Ciudad Universitaria, UNAM, DF 04510, M\'exico}
\email{rm@ciencias.unam.mx}





\begin{abstract}
Let $X$ be a smooth projective variety of dimension $d$
over an arbitrary base field $k$ and
$CH^n(X)_{\mathbb Q}$ be  the $\mathbb Q$-vector space
of codimension $n$ algebraic cycles of $X$ modulo rational equivalence,
$1\leq n \leq d$.
Consider the $\mathbb Q$-vector subspaces
$CH^n(X)_{\mathbb Q} \supseteq CH^n_{\mathrm{alg}}(X)_{\mathbb Q}
\supseteq CH^n_{\mathrm{inc}}(X)_{\mathbb Q}$
of algebraic cycles which are,
respectively, algebraically and incident
(in the sense of Griffiths) equivalent to zero.
 
Our main result computes $CH^d_{\mathrm{inc}}(X)_{\mathbb Q}$
(which coincides with the Albanese kernel $T(X)_{\mathbb Q}$
when $k$ is algebraically closed) in terms of
Voevodsky's triangulated category of  motives $DM_k$, namely, we
show that $CH^d_{\mathrm{inc}}(X)_{\mathbb Q}$ is given by
the second step of the orthogonal filtration $F^{\bullet}$ on
$CH^d(X)_{\mathbb Q}$,
i.e. $F^2 CH^d (X)_{\mathbb Q}=
CH^d_{\mathrm{inc}}(X)_{\mathbb Q}$.  
The orthogonal filtration
$F^\bullet$ on $CH^n(X)_{\mathbb Q}$ was introduced by the first author,
and is an unconditionally finite filtration satisfying several of the properties
of the still conjectural Bloch-Beilinson filtration.

We also prove that
the exterior product and intersection product of algebraic
cycles algebraically equivalent to zero is contained in the second
step of the orthogonal filtration.

Furthermore, if we assume that the field $k$ is either finite or
the algebraic closure of a finite field,
then the main result holds in any codimension, i.e. 
$F^2 CH^n_{\mathrm{alg}}(X)_{\mathbb Q}=
CH^n_{\mathrm{inc}}(X)_{\mathbb Q}$.  
We also compute in the whole Chow group, $CH^n(X)_{\mathbb Q}$,
the second step of the orthogonal filtration
$F^2 CH^n(X)_{\mathbb Q}$ in terms of the vanishing of several
intersection pairings.  
\end{abstract}

\thanks{The first author was supported by the Russian Science
Foundation grant 19-71-30002.}
 
\maketitle

\section{Introduction}  \label{sec.introd}

\subsection{}

Let $X$ be a smooth projective variety of dimension $d$ over $\mathbb C$,
$CH^n(X)$ the group of codimension $n$
algebraic cycles of $X$ modulo rational equivalence, $1\leq n\leq d$, and
$CH^n_{\mathrm{alg}}(X)\subseteq CH^n (X)$ the subgroup of algebraic cycles
which are algebraically equivalent to zero.  

One of the classic tools to study
$CH^n_{\mathrm{alg}}(X)$ is  Weil's  Abel-Jacobi map
\cite{MR0050330}*{\S IV.27}
$\psi _n: CH^n_{\mathrm{alg}}(X)\rightarrow J^n(X)$, 
where $J^n(X)$ is the
Weil-Griffiths  $n$-th intermediate Jacobian 
\cite{MR0050330}*{\S IV.24-26}, \cite{MR0233825}*{Ex. 2.1}
(see \cite{MR0233825}*{Thm. 2.54} for a comparison).
The kernel of $\psi _n$ is the subgroup 
$CH^n_{AJ}(X)\subseteq CH^n_{\mathrm{alg}}(X)$
of algebraic cycles Abel-Jacobi equivalent to zero.

In \cite{MR0309937}*{p. 6-7},
Griffiths introduced an intermediate subgroup
$CH^n_{AJ}(X)\subseteq CH^n_{\mathrm{inc}}(X)
\subseteq CH^n_{\mathrm{alg}}(X)$,
the subgroup of algebraic cycles incident equivalent to zero 
(see \ref{def.inc.equiv}),
where the quotient groups $CH^n_{\mathrm{alg}}(X)/CH^n_{AJ}(X)$ and
$CH^n_{\mathrm{alg}}(X)/CH^n_{\mathrm{inc}}(X)$
are conjecturally isogeneous 
\cite{MR0309937}*{problem B}.
The conjecture is known to be true for
$n=1,d$ (in this case both groups are equal by
the classical theory of the Picard and Albanese varieties)
 and for $n=2$  \cite{MR0815486}*{Thm. 2.4}.

In contrast to $CH^n_{AJ}(X)$, which is defined by transcendental methods, 
$CH^n_{\mathrm{inc}}(X)$  is defined in terms of the vanishing of intersection
pairings.  Thus, the subgroup of algebraic cycles incident equivalent
to zero, $CH^n_{\mathrm{inc}}(X)$, may be considered for smooth
projective varieties defined over an arbitrary base field $k$.

Hereafter we will assume that $X$ is a smooth projective variety 
of dimension $d$ over
an arbitrary base field $k$.
Our first main result \eqref{main.thm1} 
shows that $CH^d_{\mathrm{inc}}(X)_{\mathbb Q}$
is isomorphic to the second step, $F^2 CH^d (X)_{\mathbb Q}$,
of the orthogonal filtration $F^\bullet$ on $CH^d (X)_{\mathbb Q}$
 \cite{MR3614974}*{6.1.4}, \eqref{def.orth.fil.Chow}.

The orthogonal filtration on the Chow groups with coefficients
in a commutative ring $R$,
$CH^n(X)_R$, is an unconditionally finite filtration
which satisfies \cite{MR3614974}*{6.1.4}, with rational coefficients,  
several properties of the still conjectural Bloch-Beilinson-Murre filtration
\cite{MR2449178}*{conjecture 11.21},
\cite{MR923131}, \cite{MR2723320}, \cite{MR1225267}.
The first step of the filtration $F^1 CH^n(X)_R$ is given by
the $R$-submodule of algebraic cycles which are 
numerically equivalent to zero in
$CH^n(X)_R$ \cite{MR3614974}*{5.3.6}, \eqref{prop.nonperf.comp}.

Since for zero-cycles, $CH^d_{\mathrm{inc}}(X)$ is equal to 
the Albanese kernel (in case the base field $k$ is
algebraically closed),
our  result \eqref{main.thm1} 
provides further evidence \cite{MR1265533}*{Lem. 2.10}
that the orthogonal filtration on the Chow groups \cite{MR3614974}*{6.1.4},
\eqref{def.orth.fil.Chow} is
a good candidate for the Bloch-Beilinson-Murre filtration.

The construction of the
orthogonal filtration can be sketched quickly as follows.
We consider a tower 
in Voevodsky's triangulated category of motives $DM_k$:
\begin{align*} 	
	\cdots \rightarrow bc_{\leq -3}(\hr ) \rightarrow bc_{\leq -2}(\hr ) 
	\rightarrow bc_{\leq -1}(\hr )  \rightarrow \hr
\end{align*}
where $bc_{\leq m}:DM_k \rightarrow DM_k$ is a triangulated functor
\cite{MR3614974}*{3.2.3}, \cite{MR4486247}*{2.3.1}, and
$\hr$ is the motive of a point with $R$-coefficients.   Then,
the $m$-step of the filtration, $F^mCH^n (X)_R$,
 is defined as the image of the
induced map:
\begin{align*} \Hom _{DM_k}(M(X)(-n)[-2n], bc_{\leq -m}\hr ) \rightarrow &
\Hom _{DM_k}(M(X)(-n)[-2n], \hr ) 
\\ & \; \; \cong CH^{n}(X)_{R}.
\end{align*}
where the last isomorphism follows from \cite{MR1883180},
see \eqref{eq.Chow.equal.mot}.

Our second main result \eqref{main.thm3} shows that 
the exterior product and intersection product of algebraic
cycles algebraically equivalent to zero is contained in the second
step of the orthogonal filtration $F^\bullet$, i.e.
$\alpha \otimes \beta \in F^2CH^{n+m}(X\times Y)_{\mathbb Q}
\subseteq CH^{n+m}(X\times Y)_{\mathbb Q}$ for
$\alpha \in CH^n_{\mathrm{alg}}(X)_{\mathbb Q}$,
$\beta \in CH^m_{\mathrm{alg}}(Y)_{\mathbb Q}$; and
$\alpha \cdot \beta \in F^2CH^{n+m}(X)_{\mathbb Q}
\subseteq CH^{n+m}(X)_{\mathbb Q}$ for
$\alpha \in CH^n_{\mathrm{alg}}(X)_{\mathbb Q}$,
$\beta \in CH^m_{\mathrm{alg}}(X)_{\mathbb Q}$.

Our last main result \eqref{main.thm2}, where
the base field $k$ is assumed to be either finite or the algebraic closure
of a finite field, 
computes
$F^2CH^n(X)_{\mathbb Q}$, $2\leq n\leq d$, in terms of the vanishing of
several intersection pairings, and shows that
$CH^n_{\mathrm{inc}}(X)_{\mathbb Q}$ is isomorphic to the 
second step, $F^2 CH^n_{\mathrm{alg}}(X)_{\mathbb Q}$,
of the orthogonal filtration  \cite{MR3614974}*{6.1.4},
\eqref{def.orth.fil.Chow} $F^\bullet$
on $CH^n(X)_{\mathbb Q}$
restricted to $CH^n_{\mathrm{alg}}(X)_{\mathbb Q}$.

The paper is organized as follows: in section \ref{sec.preel} we introduce
the notation  and prove
some results that will be used in the rest of the paper, in section \ref{sec.mainres} we state our main results.
In section \ref{sec.firstres} we show that the second step of the
orthogonal filtration, $F^2 CH^n (X)_{\mathbb Q}$, always satisfies
the conditions \eqref{vanishing.pair1} and \eqref{vanishing.pair2},
and establish the necessary results for the proof of our
first two main theorems \eqref{main.thm1}, \eqref{main.thm3}.
In section \ref{sec.furthred}, we study the second step of the
orthogonal filtration, $F^2 CH^n (X)_{\mathbb Q}$, in terms
of Voevodsky's homotopy $t$-structure and we
establish the necessary results for
the proof of our last main theorem \eqref{main.thm2}.

\section{Preliminaries}  \label{sec.preel}

In this section we fix the notation that will be used throughout the rest
of the paper and collect together facts  from the literature
that will be necessary to establish our results.  With the exception
of \S\ref{ss.orthfil}, the results of this section are not original.

\subsection{Definitions and Notation}	\label{subsec.defandnots}		

Given a base field $k$, we
will write $Sch_{k}$ for the category of $k$-schemes
of finite type and $Sm_{k}$ for the full
subcategory of $Sch_{k}$ consisting of smooth $k$-schemes regarded
as a site with the Nisnevich topology.	  Let $SmProj_k$ be the
full subcategory of $Sm_k$ 
which consists of smooth projective $k$-schemes.
Given a field extension $K/k$
and $X,Y \in Sch_k$, let $X_K$ and $X\times _k Y$ denote, respectively,
$X\times _{\spec{k}}\spec{K}$ and $X\times _{\spec{k}}Y$.
If $X\in Sch_k$, we will write $X(k)$ for the set of $k$-points of $X$, and
$k(X)$ for its function field in case $X$
is reduced and irreducible.

We will use the following notation in all the categories under consideration: $0$ will
denote the zero object (if it exists), 
and $\cong$ will denote that a map (resp. a functor) is an
isomorphism (resp. an equivalence of categories).

We shall use freely the language of triangulated categories.  Our main references will 
be \cite{MR1812507}, \cite{MR751966}.
Given a triangulated category, we will write $[1]$ 
(resp. $[-1]$) to denote its suspension 
(resp. desuspension) functor; and for $n>0$, the composition of $[1]$
(resp. $[-1]$) iterated $n$-times will be
$[n]$ (resp. $[-n]$).  If $n=0$, then
$[0]$ will be the identity functor.

\subsection{Voevodsky's triangulated category of motives}
	
The Suslin-Voevodsky category of finite correspondences over $k$, $Cor_{k}$,
is the category with
the same objects as $Sm_{k}$
and where the morphisms $c(U,V)$ are given by the group of finite
relative cycles on $U\times _{k}V$ over $U$ \cite{MR1764199}	
 with composition as in
\cite{MR2804268}*{p. 673 diagram (2.1)}.  
The graph of a morphism in $Sm_{k}$
induces
a functor $Gr : Sm_{k}\rightarrow Cor_{k}$.  A Nisnevich sheaf with
transfers is an additive contravariant functor $\mathcal F$ from $Cor_{k}$ to the category
of abelian groups such that the restriction $\mathcal F \circ Gr$ is a Nisnevich
sheaf.  We will write $Shv^{tr}$ for the category of Nisnevich sheaves with
transfers which is an abelian category \cite{MR2242284}*{13.1}.  
For $X\in Sm_{k}$,
let $\mathbb Z _{tr}(X)$ be the Nisnevich sheaf with transfers represented by $X$
\cite{MR2242284}*{2.8 and 6.2}.

We will write $K(Shv ^{tr})$ for the category of chain complexes 
(unbounded) on $Shv ^{tr}$ equipped
with the injective model structure \cite{MR1780498}*{Prop. 3.13}, and 
$D(Shv ^{tr})$ for its homotopy category.  Let
$K^{\mathbb A ^{1}}(Shv ^{tr})$ be
the left Bousfield localization \cite{MR1944041}*{3.3} of $K(Shv ^{tr})$ with respect to the set
of maps
$\{ \mathbb Z _{tr}(X\times _{k}\mathbb A ^{1})[n]\rightarrow \mathbb Z _{tr}(X)[n]:
X\in Sm_{k}; n\in \mathbb Z\}$ induced by the projections $p:X\times _{k}\mathbb A ^{1}
\rightarrow X$. Voevodsky's  triangulated category of effective motives
$\DMeff{k}$ is the homotopy category of $K^{\mathbb A ^{1}}(Shv ^{tr})$ \cite{MR1764202}.

Let $T\in K^{\mathbb A ^{1}}(Shv ^{tr})$  be the chain complex  
$\mathbb Z_{tr}(\Gm)[1]$ \cite{MR2242284}*{2.12}, where $\Gm$ is the $k$-scheme
$\mathbb A ^1\backslash \{ 0\}$ pointed by $1$.  
We will write $Spt_{T}(Shv ^{tr})$ for the category of symmetric
$T$-spectra on $K^{\mathbb A ^{1}}(Shv ^{tr})$ equipped with the model structure
defined in \cite{MR1860878}*{8.7 and 8.11}, 
\cite{MR2438151}*{Def. 4.3.29}.  
Voevodsky's  triangulated category of  motives
$\DM{k}$ is the homotopy category of $Spt_{T}(Shv ^{tr})$ 
\cite{MR1764202}.  

Given $X\in Sm_{k}$, we will write $M(X)$ for the image of $\mathbb Z _{tr}(X)\in D(Shv^{tr})$
under the $\mathbb A ^{1}$-localization map $D(Shv ^{tr})\rightarrow \DMeff{k}$.  Let
$\Sigma ^{\infty}:\DMeff{k} \rightarrow \DM{k}$ be the suspension functor 
\cite{MR1860878}*{7.3}, 
we will abuse notation and simply write $E$ for $\Sigma
^{\infty}E$, $E\in \DMeff{k}$.  Given a map $f:X\rightarrow Y$ in $Sm _k$, we will 
write $f:M(X)\rightarrow M(Y)$ for the map induced by $f$ in $\DM{k}$.

We observe that $\DMeff{k}$ and $\DM{k}$ are tensor triangulated categories 
\cite{MR2438151}*{Thm. 4.3.76 and Prop. 4.3.77} 
with unit $\sphere = M(\spec{k})$. 
We will write $E(1)$ for $E\otimes M(\Gm)[-1]$, $E\in DM_k$ and inductively
$E(n)=(E(n-1))(1)$, $n\geq 0$.  The functor $DM_k\rightarrow DM_k$,
$E\mapsto E(1)$ is an equivalence of categories \cite{MR1860878}*{8.10}, 
\cite{MR2438151}*{Thm. 4.3.38}; 
we will write $E\mapsto E(-1)$ for
its inverse, and inductively $E(-n)=(E(-n+1))(-1)$,
$n>0$.  By convention $E(0)=E$ for $E\in \DM{k}$.

\subsubsection{} \label{coeffs.not} 
Let $R$ be a commutative ring with $1$.  We will write  
$E_{R}$ for $E\otimes \sphere _{R}$ where $E\in \DM{k}$, and $\sphere _R$
is the motive of a point with $R$-coefficients $M(\spec{k})\otimes R$.
In case the base field $k$ is non-perfect of characteristic $p$, we will always assume
that $\frac{1}{p}\in R$.

\subsubsection{Change of base field} \label{sssec.changefield}
Let $L/k$ be a field extension.  The base change functor 
$Sm_k \rightarrow Sm_L$, $X\mapsto X_L$ induces a
triangulated functor \cite{MR2438151}*{Thm. 4.5.24},
\cite{MR3590347}*{p. 298}:
\begin{align} \label{eq.ext.scalars}
    \phi ^\ast: DM_k \rightarrow DM_L
\end{align}
where $\phi ^\ast (M(X))=M(X_L)$, $X\in Sm_k$.

The following result is well-known and essentially due to Bloch.
\begin{prop} \label{torsion.ext.scalars}
With the notation and conditions of  \eqref{coeffs.not}.
Let $E$, $F\in DM_k$.
The kernel of the map induced by \eqref{eq.ext.scalars} is torsion:
\begin{align}
	\xymatrix{\phi^\ast: 
	\Hom_{DM_k}(E,F) \ar[r]& \Hom_{DM_L}(\phi ^\ast E,\phi ^\ast F)
	}
\end{align}
\end{prop}
\begin{proof}
By \cite{MR3590347}*{Cor. 4.13, Thm. 4.12} we may assume that $k$ is 
a perfect field.
Then the result follows from \cite{MR2723320}*{Lem. 1A.3}
 mutatis-mutandis.
\end{proof}

\subsection{Voevodsky's slice filtration}  \label{ss.slicefil}

The triangulated category of motives,
$\DM{k}$, is a compactly generated triangulated category 
\cite{MR1308405}*{Def. 1.7} with
compact generators \cite{MR2438151}*{Thm. 4.5.67}:
\begin{align}  \label{eq.DMgens}
 \generators _{DM_k}=\{ M(X)(r): X\in Sm_{k}; r \in \mathbb Z\}.
\end{align}

For $m\in \mathbb Z$, we consider:
\begin{align}  \label{eq.DMeffgenstw}
 \generators ^{\mathrm{eff}}(m)=\{ M(X)(r): X\in Sm_{k}; r\geq m\}
 \subseteq \generators _{DM_k}.
\end{align}

\subsubsection{}  \label{subsubsec.neffDM}
The slice filtration \cite{MR1977582},
\cite{MR2600283}*{p. 18}, \cite{MR2249535}
is the following tower of
triangulated subcategories of $DM_k$:
\begin{align}  \label{eq.slice.filtration}
\cdots \subseteq \DMeff{k} (m+1)\subseteq \DMeff{k} (m) \subseteq 
\DMeff{k} (m-1)\subseteq \cdots
\end{align}
where
$\DMeff{k} (m)$ is the smallest 
full triangulated subcategory of  $\DM{k}$ which
contains $\generators ^{\mathrm{eff}}(m)$ \eqref{eq.DMeffgenstw}
and is closed under arbitrary (infinite) coproducts.  
Notice that $\DMeff{k} (m)$ is compactly generated with set of generators
$\generators ^{\mathrm{eff}}(m)$ \cite{MR1308405}*{Thm. 2.1(2.1.1)}.

\subsubsection{} \label{sss.eff.cov}
By \cite{MR1308405}*{Thm. 4.1}
the inclusion $i_m:\DMeff{k} (m)\rightarrow DM_k$ admits a right adjoint
$r_m:DM_k \rightarrow \DMeff{k} (m)$ which is also a triangulated functor.
The $(m-1)$-effective cover of the slice filtration is defined to be
$f_m=i_m \circ r_m :DM_k \rightarrow DM_k$ \cite{MR1977582},
\cite{MR2600283}*{p. 18}, \cite{MR2249535}.

\begin{prop}  \label{prop.func.sfil}
With the notation and conditions of  
\eqref{coeffs.not}-\eqref{sssec.changefield}.
Let $L/k$ be any field extension.
Then, for any $m\in \mathbb Z$,
$f_m\circ \phi ^\ast \cong \phi ^\ast \circ f_m$
 \eqref{eq.ext.scalars}, \eqref{sss.eff.cov}.
\end{prop}
\begin{proof}
The result follows by
combining \cite{MR2600283}*{Lem. 5.9}, \cite{MR2249535}*{Prop. 1.1}
with \cite{MR3590347}*{Thm. 4.12 and Thm. 5.1}.  See also
\cite{MR3034283}*{2.12-2.13}.
\end{proof}

\begin{prop}  \label{prop.perfclosure}
With the notation and conditions of  \eqref{coeffs.not}-\eqref{sssec.changefield}.
Let $k$ be a non-perfect field and $L$ its perfect closure.
Let $h:E\rightarrow F$ be a map  in $DM_k$ and $m\in \mathbb Z$.
Consider the base change functor $\phi ^\ast: DM_k \rightarrow DM_L$
 \eqref{eq.ext.scalars}.
Then, $f_m(h)=0$ in $DM_k$ if and only if $f_m(\phi ^\ast h)=0$ in $DM_L$.
\end{prop}
\begin{proof}
We observe that $\phi ^\ast$ is an equivalence of categories
\cite{MR3590347}*{Cor. 4.13}.  Thus it suffices to see that
$f_m\circ \phi ^\ast \cong \phi ^\ast \circ f_m$, which follows from
\eqref{prop.func.sfil}.
\end{proof}

\subsection{}  \label{subsubsec.cancel}

With the notation and conditions of  \eqref{coeffs.not}.
By Voevodsky's cancellation theorem 
\cite{MR2804268}*{Cor. 4.10} (combined with 
\cite{MR3590347}*{Cor. 4.13, Thm. 4.12 and Thm. 5.1} in case 
$k$ is non-perfect), 
the suspension functor 
$\Sigma ^{\infty}:\DMeff{k} \rightarrow DM_k$ induces an equivalence of 
categories between 
$\DMeff{k}$ and $\DMeff{k} (0)$
\eqref{subsubsec.neffDM}.
We will abuse notation and write $\DMeff{k}$ for 
$\DMeff{k} (0)$. 

\subsubsection{Motivic cohomology}  \label{sssec.motcoh}
Let $X\in Sm_k$, and $r$, $s\in \mathbb Z$.
We will write $H^{r,s}(X,R)$ for 
$\Hom _{DM_k}(M(X),\sphere _R (s)[r])$,
i.e. for the motivic cohomology of $X$ with coefficients in $R$ of degree $r$ and
weight $s$.

Combining \eqref{subsubsec.cancel} and \cite{MR1883180} we conclude
that  there are natural
isomorphisms:
\begin{align} \label{eq.Chow.equal.mot}
\Hom_{DM_k}(M(X)(-s)[-r], \sphere _R) \xrightarrow[\cong]{A}
H^{r,s}(X,R) \xrightarrow[\cong]{V} CH^s(X, 2s-r)_R,
\end{align}
where $A(\alpha)= \alpha (s)[r]$ and the groups
on the right are Bloch's higher Chow groups \cite{MR0852815}.

\subsubsection{Lieberman's lemma \cite{MR0382267}*{p. 73}, 
\cite{MR3052734}*{Lem. 2.1.3}} \label{sssec.Lieblem}
Let $X \in SmProj_k$ of dimension $d$, 
$Y\in Sm_k$ and $a$, $b$, $r$, $s\in \mathbb Z$.
We will write $\pi _X:X\times Y\rightarrow X$, $\pi _Y:X\times Y \rightarrow Y$
for the projections.

Consider the following composable maps in $DM_k$,
$\alpha : M(X)(-s)[-r]\rightarrow \sphere _R$,
$\beta :M(Y)(-b)[-a]\rightarrow M(X)(-s)[-r]$.  Let
$\alpha ^A =\alpha (s)[r]\in H^{r,s}(X,R)$ \eqref{eq.Chow.equal.mot} and
$\beta ^P \in H^{2d+a-r,d+b-s}(X\times Y, R)$ be the image of $\beta$
under the isomorphism induced by dualizing $M(X)$ 
\cite{MR1764202}*{Thm. 4.3.7},
\cite{MR2399083}*{Prop. 6.7.1 and \S 6.7.3}:
\begin{align*}
\Hom_{DM_k}&(M(Y)(-b)[-a],M(X)(-s)[-r])\\ &\xrightarrow[\cong]{}
\Hom_{DM_k}(M(X\times Y), \sphere_R(d+b-s)[2d+a-r])
\end{align*}

\begin{prop}  \label{prop.Lieb-lem}
With the notation and conditions of 
\eqref{coeffs.not} and
\eqref{eq.Chow.equal.mot}-\eqref{sssec.Lieblem}.
Then:
\begin{align}
\pi _{Y\ast}((\pi _X ^\ast \alpha ^A) \cdot \beta ^P) &=
(\alpha \circ \beta)(b)[a]\;
&\mathrm{ in }\; &H^{a,b}(Y,R) \label{prop.Lieb-lem.a}\\
\pi _{Y\ast}((\pi _X ^\ast V(\alpha ^A)) \cdot V(\beta ^P)) &=
V((\alpha \circ \beta)(b)[a])\;
&\mathrm{ in }\; &CH^{b}(Y,2b-a)_R \label{prop.Lieb-lem.b}
\end{align}
where $\cdot$ is the product on motivic cohomology (resp. the higher Chow groups)
and $\pi _X ^\ast$, $\pi _{Y\ast}$ are the pull-back and push-forward
in motivic cohomology (resp. the higher Chow groups).
\end{prop}
\begin{proof}
On the one hand,
\eqref{prop.Lieb-lem.a} follows directly from the formal argument in
\cite{MR1623774}*{Ch. IV, Lem. 2.1.2}.  Notice that in contrast to Voevodsky's construction, the construction
in \cite{MR1623774} is contravariant.

On the other hand,
\eqref{prop.Lieb-lem.b} follows by combining \eqref{prop.Lieb-lem.a} with
\cite{MR1743246}*{Thm. 0.2} and \cite{MR2738368}*{Thm. 3.1}.
\end{proof}

\begin{prop}  \label{prop.tenscoef}
Let $Y\in Sm_k$ and $r$, $s\in \mathbb Z$.  Assume further
that $R=\mathbb Z [\tfrac{1}{p}]$ or $R=\mathbb Q$ \eqref{coeffs.not}.
Then,
\begin{align*}
\Hom _{\DM{k}}(M(Y)(s)[r], E_R)\cong
\Hom _{\DM{k}}(M(Y)(s)[r], E)\otimes _{\mathbb Z} R.
\end{align*}
\end{prop}
\begin{proof}
We observe that $M(Y)(s)[r]$ is compact in $\DM{k}$
\cite{MR2438151}*{Thm. 4.5.67}.  Then the result follows from
\cite{MR1308405}*{Lem. 2.8}, since $E_R$ is given by
the homotopy colimit \cite{MR1812507}*{1.6.4} 
of the following diagrams in
$\DM{k}$:
\begin{align*}
\xymatrix{E\otimes \mathbb Z [\frac{1}{p}]\cong  \hocolim (
E \ar[r]^-{p\cdot id_E}& E \ar[r]^-{p^2\cdot id_E}& \cdots \ar[r]&
E \ar[r]^-{p^n \cdot id_E} &
E \ar[r]&\cdots ) \\
E\otimes \mathbb Q \cong \hocolim (
E \ar[r]^-{2\cdot id_E}& E \ar[r]^-{3\cdot id_E}& \cdots \ar[r]&
E \ar[r]^-{n \cdot id_E} &
E \ar[r]&\cdots )}
\end{align*}
\end{proof}

\begin{prop}  \label{prop.vanchancoef}
With the notation and conditions of \eqref{coeffs.not}.
Let $h:E\rightarrow F_R$ be a map in $\DM{k}$.  Assume further that
$R=\mathbb Z [\tfrac{1}{p}]$ or $R=\mathbb Q$.  Then, $h=0$ in
$\DM{k}$ if and only if 
$0=h\otimes R : E\otimes R \rightarrow F_R \otimes R$ in $\DM{k}$.
\end{prop}
\begin{proof}
If $h=0$, then it's immediate that $h\otimes R=0$, since
$-\otimes R: \DM{k}\rightarrow \DM{k}$, $E\mapsto E\otimes R$
is a triangulated functor.

On the other hand, if $h\otimes R =0$, then we deduce that
$h=0$ from the following commutative diagram in $\DM{k}$:
\begin{align*}
\xymatrix{ E\cong E\otimes \sphere \ar[rr]^-{E\otimes u_R} \ar[d]_-{h} && 
E\otimes R \ar[d]^-{h\otimes R} \\
F\otimes R \cong (F\otimes R)\otimes \sphere
 \ar[rr]^-{F_R\otimes u_R}_-{\cong} && (F\otimes R)\otimes R}
\end{align*}
where the bottom isomorphism follows from \eqref{prop.tenscoef},
and $u_R:\sphere \rightarrow R$ is the unit map. 
\end{proof}

\begin{prop}  \label{prop.slicechangcof}
With the notation and conditions of \eqref{coeffs.not}.
Let $E\in \DM{k}$ and $m\in \mathbb Z$.
Assume further that $R=\mathbb Z [\tfrac{1}{p}]$ or $R=\mathbb Q$.
Then, $f_m(E)_R\cong f_m(E_R)$ in $\DM{k}$.
\end{prop}
\begin{proof}
By \eqref{prop.tenscoef}, we notice that the triangulated functor
$-\otimes R:\DM{k}\rightarrow \DM{k}$ satisfies the
condition \cite{MR3034283}*{2.11}, so the result follows by
\cite{MR3034283}*{2.13 and 2.12}.
\end{proof}

\subsection{The orthogonal filtration}   \label{ss.orthfil}
Let $\northogonal{k}{m}$, $m\in \mathbb Z$ be the full
subcategory of $DM_k$ which consists of objects $E\in DM_k$
such that for every $G\in \DMeff{k}(m)$
\eqref{subsubsec.neffDM}: $\Hom _{DM_k}(G,E)=0$.

\subsubsection{}
The orthogonal filtration \cite{MR3614974}*{3.2.1}
(called birational in \cite{MR3614974}*{3.2})
is the following tower of  triangulated subcategories of $DM_k$:
\begin{align}  \label{eq.orthogonal.filtration}
\cdots \subseteq \northogonal{k}{m-1}\subseteq \northogonal{k}{m} \subseteq 
\northogonal{k}{m+1} \subseteq \cdots
\end{align}

\subsubsection{}  \label{def.orthogonal.adj}
We observe that the  inclusion $j_m:\northogonal{k}{m}\rightarrow DM_k$ admits
a right adjoint $p_m: DM_k \rightarrow \northogonal{k}{m}$ which is also
a triangulated functor \cite{MR1308405}*{Thm. 4.1},
\cite{MR3614974}*{3.2.2}.  Let $bc_{\leq m}=j_{m+1}\circ p_{m+1}:DM_k \rightarrow
DM_k$,
where $bc$ stands for birational cover \cite{MR3614974}*{3.2}.
The counit of the adjunction, $\theta _m :bc_{\leq m} \rightarrow id$,
 satisfies a universal property \cite{MR3614974}*{3.2.4},
which together with the inclusions \eqref{eq.orthogonal.filtration} induces
a canonical natural transformation $bc_{\leq m}\rightarrow
bc_{\leq m+1}$.

\subsubsection{} 
For $\sphere _R \in DM_k$ \eqref{coeffs.not} and 
$m\geq 0$, the counit
$\theta _m ^{\sphere _R}:bc_{\leq m}\sphere _R \rightarrow \sphere _R$
is an isomorphism  \cite{MR3614974}*{5.1.2}.  
Thus, we obtain the following tower
in $DM_k$:
\begin{align} \label{eq.orth.tower.unit}
\begin{split}
\xymatrix{\cdots \ar[r] & bc_{\leq -m}\sphere _R \ar[r] \ar[drrr] |-{\theta _{-m}^{\sphere _R}}
& \cdots \ar[r]
& bc_{\leq -2}\sphere _R
\ar[r] \ar[dr]|-{\theta _{-2}^{\sphere _R}}& 
bc_{\leq -1}\sphere _R \ar[d] |-{\theta _{-1}^{\sphere _R}}\\
&&&& \sphere_R}
\end{split}
\end{align}

\begin{defi} \label{def.orth.fil.Chow}
Let $X\in SmProj_k$ of dimension $d$, $1\leq n\leq d$,
and with
the notation and conditions of  \eqref{coeffs.not}.
The \emph{orthogonal filtration} on $CH^{n}(X)_R$ 
\cite{MR3614974}*{5.2.1}
is the decreasing filtration $F^\bullet$, where
$F^mCH^{n}(X)_R$, $m\geq 0$ is the image
of $\theta _{-m} ^{\sphere _R}$ \eqref{eq.orth.tower.unit}:
\begin{align*}
\xymatrix{\Hom_{DM_k}(M(X)(-n)[-2n], bc_{\leq -m}\sphere _R)
\ar[d]_-{\theta _{-m \ast}^{\sphere _R}}\\
\Hom_{DM_k}(M(X)(-n)[-2n], \sphere _R)\cong CH^n(X)_R
}
\end{align*}
where the bottom isomorphism is given by \eqref{eq.Chow.equal.mot}.
\end{defi}

\subsubsection{}  \label{sss.main.not}
With the notation and conditions of 
\eqref{coeffs.not}.
Let $\alpha \in CH^n (X)_R$, $X\in SmProj_k$.
We will abuse notation and also write
$\alpha: M(X)\rightarrow \sphere _R (n)[2n]$ for the map that
corresponds to
$\alpha \in CH^n(X)_R$ under 
the isomorphism $V$ in \eqref{eq.Chow.equal.mot}, 
\eqref{sssec.motcoh}.

\begin{rmk}  \label{rmk.orth.crit}
With the notation and conditions of 
\eqref{def.orth.fil.Chow}-\eqref{sss.main.not}.
Then, 
$\alpha \in F^mCH^{n}(X)_R$
if and only if $f_{-m+1}(\alpha (-n)[-2n])=0$ \eqref{sss.eff.cov}.

In effect, this follows directly from the first part of \cite{MR3614974}*{5.3.3}
and the universal property of the counit 
$\epsilon _{-m+1}^{\sphere _R}:f_{-m+1}\sphere _R\rightarrow \sphere _R$
\cite{MR3614974}*{3.3.1}.
\end{rmk}

\subsubsection{}
The orthogonal filtration on the Chow groups \eqref{def.orth.fil.Chow}
satisfies several of the properties of the still
conjectural Bloch-Beilinson-Murre filtration 
\cite{MR3614974}*{6.1.4}.  Strictly speaking, \cite{MR3614974}*{6.1.4}
is only stated for perfect base fields, but by \eqref{prop.nonperf.comp} the result
holds as well for non-perfect base fields  if we consider
$\mathbb Z [\frac{1}{p}]$-coefficients \eqref{coeffs.not},
 since \cite{MR3614974}*{6.1.1($\mathbf{BBM2}$)}
is the only property in \cite{MR3614974}*{6.1.4} that does not follow
from the construction.

\begin{lem}  \label{lem.nonperf-perf.comp}
With the notation and conditions of 
\eqref{def.orth.fil.Chow}-\eqref{sss.main.not}.
Let $k$ be a non-perfect base field and $L$ its perfect closure.
Then, $\alpha \in F^m CH^n (X)_R$ 
\eqref{def.orth.fil.Chow} if and only if
$\alpha _L \in F^m CH^n (X_L)_R$.
\end{lem}
\begin{proof}
It follows by combining \eqref{rmk.orth.crit} with
\eqref{prop.perfclosure}.
\end{proof}

\begin{lem}  \label{lem.fieldext.comp}
With the notation and conditions of 
\eqref{def.orth.fil.Chow}-\eqref{sss.main.not}.
Let $L/k$ be any field extension.
Then, $\alpha \in F^m CH^n (X)_{\mathbb Q}$ 
\eqref{def.orth.fil.Chow} if and only if
$\alpha _L \in F^m CH^n (X_L)_{\mathbb Q}$.
\end{lem}
\begin{proof}
Consider the base change functor $\phi ^\ast: DM_k \rightarrow DM_L$
 \eqref{eq.ext.scalars}.
We observe that
$f_{-m+1}\circ \phi ^\ast \cong \phi ^\ast \circ f_{-m+1}$ \eqref{prop.func.sfil}.
Then, the result follows
by combining \eqref{rmk.orth.crit} with
\eqref{torsion.ext.scalars}.
\end{proof}

\begin{prop}  \label{prop.nonperf.comp}
Let $k$ be a non-perfect base field, and with
the notation and conditions of  \eqref{def.orth.fil.Chow}.
 Then, $F^1CH^n(X)_R$ \eqref{def.orth.fil.Chow}
is the $R$-submodule of $CH^n (X)_R$ of cycles numerically equivalent to zero.
\end{prop}
\begin{proof}
Let $L$ be the perfect closure of $k$.  Since the exponential characteristic
of $k$ is a unit in $R$ \eqref{coeffs.not} we conclude that
$\alpha \in CH^n (X)_R$ is numerically equivalent to zero if and only if
$\alpha _L \in CH^n (X_L)_R$ is numerically equivalent to zero.

Thus, by \cite{MR3614974}*{5.3.6} it suffices to see that $\alpha \in F^1CH^n(X)_R$
if and only if $\alpha _L \in F^1CH^n(X_L)_R$, which follows by 
\eqref{lem.nonperf-perf.comp}.
\end{proof}

\section{The second step of the orthogonal filtration}  \label{sec.mainres}

\subsection{Incidence equivalence}  \label{sec.inc.equiv}

With the notation and conditions of  
\eqref{def.orth.fil.Chow}-\eqref{sss.main.not}.
The vanishing of the following
pairings is central for the rest of the paper:

\subsubsection{}  \label{vanishing.pair1}
For every $Y\in SmProj_k$, and every 
$\beta \in CH^{d-n+1}(X\times Y)_R$:
\begin{align*}
	\pi _{Y\ast}((\pi _X ^\ast \alpha)\cdot \beta)=0 \in CH^1(Y)_R,
\end{align*}
where $\pi_X:X\times Y \rightarrow X$, $\pi _Y:X\times Y\rightarrow Y$
are the projections.

\begin{rmk}  \label{rmk.car.vanpair1}
With the notation and conditions of  
\eqref{sec.inc.equiv}.
\begin{enumerate}
\item  \label{rmk.car.vanpair1.a}
By the projective bundle formula, if $\alpha \in CH^n(X)_R$ 
satisfies \eqref{vanishing.pair1} then $\alpha$ is
numerically equivalent to zero.

\item \label{rmk.car.vanpair1.b}
If $\alpha \in CH^1(X)_R$ satisfies \eqref{vanishing.pair1},
then $\alpha =0$.

In effect, consider in \eqref{vanishing.pair1}:
$Y=X$ and $\beta =\Delta _X \in CH^d(X\times X)_R$
the class of the diagonal.

\item \label{rmk.car.vanpair1.c}
Assume further that the base field $k$ is algebraically closed.  
Let $\alpha \in CH^d(X)_{R}$.  Then, by the theory of 
divisorial correspondences
\cite{MR106225}*{p.155 Thm. 2},
\cite{MR1265529}*{Thm. 3.9}:
$\alpha$ satisfies \eqref{vanishing.pair1}
if and only if $\alpha$ is in the Albanese kernel $T(X)_{R}
\subseteq CH^d(X)_{R}$.

\end{enumerate}
\end{rmk}

\begin{defi}[Griffiths  \cite{MR0309937}*{p.6-7}] \label{def.inc.equiv}
With the notation and conditions of 
\eqref{def.orth.fil.Chow}-\eqref{sss.main.not}.
We will say that $\alpha \in CH^n(X)_R$ is 
\emph{incident equivalent to zero}
if $\alpha \in CH^n_{\mathrm{alg}}(X)_R$, i.e. it is algebraically equivalent 
to zero, and satisfies the condition 
\eqref{vanishing.pair1}.

We will write $CH^n_{\mathrm{inc}}(X)_R\subseteq CH^n(X)_R$ 
for the $R$-submodule of
algebraic cycles which are incident equivalent to zero.
\end{defi}

\begin{rmk}  \label{rmk.vanpair1eqinceq}
We observe that for zero cycles, by 
\ref{rmk.car.vanpair1}\eqref{rmk.car.vanpair1.a},
$\alpha \in CH^d(X)_R$ satisfies \eqref{vanishing.pair1} if
and only if $\alpha \in CH^d_{\mathrm{inc}}(X)_R$.
\end{rmk}

\subsection{}

With the notation and conditions of 
\eqref{def.orth.fil.Chow}-\eqref{sss.main.not}.
We also need to consider the vanishing of the following pairings
which where studied by Bloch \cite{MR991974}*{p. 21}:

\subsubsection{}  \label{vanishing.pair2}
For every $Y\in Sm_k$, and every 
$\beta \in CH^{d-n+1}(X\times Y,1)_R$:
\begin{align*}
	\pi _{Y\ast}((\pi _X ^\ast \alpha)\cdot \beta)=0 \in 
	CH^1(Y,1)_R\cong \Gamma (Y, \mathcal O ^\ast _Y)_R,
\end{align*}
where $\pi_X:X\times Y \rightarrow X$, $\pi _Y:X\times Y\rightarrow Y$
are the projections.

\begin{rmk}  \label{rmk.hompair2van}
With the notation and conditions of 
\eqref{def.orth.fil.Chow}-\eqref{sss.main.not}.
Assume further that
$\alpha \in CH^n(X)_R$ is homologically equivalent to zero 
in the sense of \cite{MR991974}*{p. 21}.  Then, by 
\cite{MR991974}*{Lem. 1}
the condition
\eqref{vanishing.pair2} holds. 
\end{rmk}

\begin{lem}  \label{lem.nonperf.incequiv}
With the notation and conditions of 
\eqref{def.orth.fil.Chow}-\eqref{sss.main.not}.
Let $k$ be a non-perfect base field and $L$ its perfect closure.
Then, $\alpha \in CH^n (X)_R$ satisfies \eqref{vanishing.pair1}
(resp. \eqref{vanishing.pair2})
if and only if
$\alpha _L \in CH^n (X_L)_R$ satisfies \eqref{vanishing.pair1}
(resp. \eqref{vanishing.pair2}).
\end{lem}
\begin{proof}
Since the exponential characteristic
of $k$ is a unit in $R$ \eqref{coeffs.not},
the result follows from \cite{MR3590347}*{Thm. 1.11 and Lem. 1.12}.
\end{proof}

\begin{lem}  \label{lem.perf.incequiv}
With the notation and conditions of 
\eqref{def.orth.fil.Chow}-\eqref{sss.main.not}.
Let $k$ be a perfect base field and $L$ its algebraic closure.
Then, $\alpha \in CH^n (X)_{\mathbb Q}$ 
satisfies \eqref{vanishing.pair1}
(resp. \eqref{vanishing.pair2})
if and only if
$\alpha _L \in CH^n (X_L)_{\mathbb Q}$ satisfies \eqref{vanishing.pair1}
(resp. \eqref{vanishing.pair2}).
\end{lem}
\begin{proof}
Assume that $\alpha _L \in CH^n (X_L)_{\mathbb Q}$ satisfies \eqref{vanishing.pair1}
(resp. \eqref{vanishing.pair2}).  Then, combining \eqref{torsion.ext.scalars}
with \eqref{sss.appl.Lieblem} we conclude that
$\alpha \in CH^n (X)_{\mathbb Q}$ also satisfies
\eqref{vanishing.pair1}
(resp. \eqref{vanishing.pair2}).

Now, assume that $\alpha \in CH^n (X)_{\mathbb Q}$ 
satisfies \eqref{vanishing.pair1}
(resp. \eqref{vanishing.pair2}).  Since the extension
$L/k$ is algebraic, we conclude that
for every $Y\in Sm_L$ there exist: a finite field extension
$k'/k$, $Y'\in Sm_{k'}$, such that $Y'_{L}\cong Y$.  We
observe that $Y'\in Sm_k$, since $k$ is perfect. 
Hence, 
\begin{align*}
 CH^{d-n+1}(X_L\times Y)_{\mathbb Q} &=\colim 
 _{k'\subseteq k''\subset L}\;
CH^{d-n+1}(X_{k''}\times Y'_{k''})_{\mathbb Q} \\
CH^{d-n+1}(X_L\times Y,1)_{\mathbb Q} &=\colim 
_{k'\subseteq k''\subset L}\;
CH^{d-n+1}(X_{k''}\times Y'_{k''},1)_{\mathbb Q}
\end{align*}
where $k''/k'$ is a finite field extension,
and $X_{k''}\times Y'_{k''}\in Sm_k$.  Thus, we conclude
that $\alpha _L \in CH^n (X_L)_{\mathbb Q}$ satisfies \eqref{vanishing.pair1}
(resp. \eqref{vanishing.pair2}).
\end{proof}

\begin{lem}  \label{lem.algeq.vanpair2}
With the notation and conditions of 
\eqref{def.orth.fil.Chow}-\eqref{sss.main.not}.
Let $k$ be an arbitrary base field.  Assume further
that $\alpha \in CH^n _{\mathrm{alg}}(X)_{\mathbb Q}$ \eqref{def.inc.equiv}.
Then, $\alpha$ 
satisfies \eqref{vanishing.pair2}.
\end{lem}
\begin{proof}
We notice that
 $N\alpha \in CH^n_{\mathrm{alg}}
(X)$  for some $N\in \mathbb Z$.  
Thus, $\pi _X ^\ast (N\alpha) \in CH^n_{\mathrm{alg}}
(X\times Y)$
\eqref{vanishing.pair2}, so $\pi _X ^\ast (N\alpha) \in CH^n 
(X\times Y)$
is homologically equivalent to zero in the sense of 
\cite{MR991974}*{p. 21}.
Then, by \cite{MR991974}*{Lem. 1},
we conclude that  \eqref{vanishing.pair2} holds for $N\alpha$, 
which implies that
 \eqref{vanishing.pair2} holds
for $\alpha$ and $R=\mathbb Q$.
\end{proof}

\subsection{Main results}

\begin{thm} \label{main.thm1}
Let $k$ be an arbitrary base field, $X\in SmProj _k$ of dimension $d$,
and $\alpha \in CH^d(X)_{\mathbb Q}$.
Then, $\alpha \in F^2CH^d(X)_{\mathbb Q}$ 
\eqref{def.orth.fil.Chow} if and only if
\eqref{vanishing.pair1} holds.  Hence,
$F^2CH^d(X)_{\mathbb Q} \cong
CH^d_{\mathrm{inc}}(X)_{\mathbb Q}$,
and if we assume further that $k$ is algebraically closed, then
$F^2CH^d(X)_{\mathbb Q} \cong T(X)_{\mathbb Q}$, the 
Albanese kernel.
\end{thm}
\begin{proof}
If $d=1$, we observe, by 
\ref{rmk.car.vanpair1}\eqref{rmk.car.vanpair1.b} that,
 $CH^1_{\mathrm{inc}}(X)_{\mathbb Q}=0$.  So, in this case
 the result follows from \cite{MR3614974}*{6.1.4}.  Thus,
 we may assume that $d\geq 2$.
 
 Now, if $\alpha \in F^2 CH^d (X)_{\mathbb Q}$,  
 then the result
follows from \eqref{prop.pf.nec.cond}.

On the other hand, if $\alpha \in CH^d (X)_{\mathbb Q}$ satisfies
\eqref{vanishing.pair1},
combining \eqref{lem.nonperf.incequiv} with \eqref{lem.nonperf-perf.comp},
we conclude that it is enough to consider the case of a perfect base
field $k$.

Now, let $k$ be a perfect field and $\bar{k}$ its algebraic closure.
By \eqref{lem.perf.incequiv}, $\alpha \in CH^d(X)_{\mathbb Q}$ satisfies
\eqref{vanishing.pair1} 
if and only if $\alpha_{\bar{k}} \in CH^n(X_{\bar{k}})_{\mathbb Q}$
satisfies \eqref{vanishing.pair1}.
Thus, considering
$L= \bar{k}$  in \eqref{lem.fieldext.comp}, we deduce
that it suffices to prove the result when $k$ is
algebraically closed.

So, we consider the case of an algebraically closed base field $k$.
If $k$ is not the algebraic closure of a finite field, 
we conclude
that $\alpha \in F^2 CH^d (X)_{\mathbb Q}$,
by the functoriality of the orthogonal
filtration \cite{MR3614974}*{6.1.4}, \eqref{def.orth.fil.Chow} and
combining \ref{rmk.car.vanpair1}\eqref{rmk.car.vanpair1.c},
 \eqref{prop.zercyd4}, \eqref{prop.3fold} and \eqref{prop.surface}.
 If $k$ is the algebraic closure of a finite field, we observe that
\eqref{rmk.vanpair1eqinceq} and
 \eqref{lem.algeq.vanpair2} imply that
 $\alpha \in CH^d (X)_{\mathbb Q}$ also satisfies \eqref{vanishing.pair2},
 and then we conclude that $\alpha \in F^2 CH^d (X)_{\mathbb Q}$
 by combining \eqref{prop.0-1hom.vanpair1-2}
 with \eqref{prop.alpha1ff}.
 
 Then, the isomorphism $F^2CH^d(X)_{\mathbb Q} \cong
CH^d_{\mathrm{inc}}(X)_{\mathbb Q}$ follows from
\eqref{rmk.vanpair1eqinceq}, and $CH^d_{\mathrm{inc}}(X)_{\mathbb Q}
=T(X)_{\mathbb Q}$ follows from 
the definition \eqref{def.inc.equiv}
and \ref{rmk.car.vanpair1}\eqref{rmk.car.vanpair1.c}.
\end{proof}

\begin{thm} \label{main.thm3}
Let $k$ be an arbitrary base field,  
$X$, $Y\in SmProj _k$ of dimension $d$, $e$,
respectively;
and $\alpha \in CH^n_{\mathrm{alg}}(X)_{\mathbb Q}$,
$\beta \in CH^m_{\mathrm{alg}}(Y)_{\mathbb Q}$.
Then, 
\begin{enumerate}
\item \label{main.thm3.a}
the exterior product: $\alpha \otimes \beta \in F^2 CH^{n+m}
(X\times Y)_{\mathbb Q} \subset CH^{n+m}(X\times Y)_{\mathbb Q}$,
\item \label{main.thm3.b}
assume further that  $Y=X$, then the intersection product: 

$\alpha \cdot \beta \in F^2 CH^{n+m}
(X)_{\mathbb Q} \subset CH^{n+m}(X)_{\mathbb Q}$.
\end{enumerate}
\end{thm}
\begin{proof}
\eqref{main.thm3.a}:  
We notice that, by \eqref{lem.fieldext.comp}, it is enough to prove
the result when $k$ is algebraically closed.  Now,
since $\alpha$ and $\beta$ are algebraically equivalent to zero,
there exist curves:
$C$, $C'\in SmProj_k$, zero cycles:
$\gamma _{\alpha}\in CH^1_{\mathrm{alg}}(C)_{\mathbb Q}$,
$\gamma _{\beta}\in CH^1_{\mathrm{alg}}(C')_{\mathbb Q}$,
and Chow correspondences:
$\Lambda _{\alpha}\in CH^{n}(X\times C)_{\mathbb Q}$,
$\Lambda _{\beta}\in CH^{m}(Y\times C')_{\mathbb Q}$, such that:
\begin{align*}
\alpha = p_{X\ast}(\Lambda _{\alpha}\cdot (p_C^\ast \gamma _{\alpha})),\\
\beta = p_{Y\ast}(\Lambda _{\beta}\cdot (p_{C'}^\ast \gamma _{\beta})),
\end{align*}
where $p_X:X\times C \rightarrow X$, $p_C:X\times C \rightarrow C$,
$p_Y:Y\times C' \rightarrow Y$ and $p_{C'}:Y\times C'$ are the
projections.

Thus,
$\alpha \otimes \beta =p_{X\times Y \ast}((\Lambda_{\alpha}\otimes
\Lambda _{\beta})\cdot (p_C\times p_{C'})^\ast (\gamma _{\alpha}\otimes
\gamma _{\beta}))$.  So, by the functoriality of the orthogonal
filtration \cite{MR3614974}*{6.1.4}, \eqref{def.orth.fil.Chow} 
it suffices to show that
$\gamma_{\alpha}\otimes \gamma _{\beta}\in F^2 CH^2(C\times C')
_{\mathbb Q}$.

We observe that
$\gamma _{\alpha}\otimes \gamma_{\beta}\in
CH^2_{\mathrm{alg}}(C\times C')_{\mathbb Q}$, so
\eqref{main.thm1} 
implies that it suffices to show that
$\gamma _{\alpha}\otimes \gamma_{\beta} \in
T(C\times C')_{\mathbb Q}$, the Albanese kernel of
$C\times C'$.  This is well known.

We provide the details.
Let $J(C)$, $J(C')$ be the Jacobian of $C$, $C'$,
respectively.  Thus, $J(C)(k)\cong CH^1_{\mathrm{alg}}(C)$ and
$J(C')(k)\cong CH^1_{\mathrm{alg}}(C')$, so, composing
the exterior product, $E$, and the Albanese morphism
we obtain the following diagram:
\begin{align*}
\xymatrix{J(C)(k)\times J(C')(k) \ar[r]^-{E}& 
CH^2_{\mathrm{alg}}(C\times C')
\ar[r]^-a& Alb(C\times C')(k)}
\end{align*}
By rigidity 
\cite{MR106225}*{p.22 Lem. 2}: $a\circ E=0$, since
$E(J(C)(k)\times \{0 \})=0=E(\{0 \} \times J(C')(k))$.  Thus,
$\gamma _{\alpha}\otimes \gamma _{\beta}\in 
T(C\times C')_{\mathbb Q}$, 
since $\gamma _{\alpha}\otimes \gamma _{\beta}$ is in the image
of $E$.

\eqref{main.thm3.b}: The result follows from 
\ref{main.thm3}\eqref{main.thm3.a} and the functoriality of the
orthogonal filtration \cite{MR3614974}*{6.1.4}, since
$\alpha \cdot \beta = \Delta _X ^\ast (\alpha \otimes \beta)$,
where $\Delta _X: X\rightarrow X\times X$ is the diagonal
embedding.
\end{proof}

\begin{thm} \label{main.thm2}
Let $k$ be a base field which is either finite or the
algebraic closure of a finite field, 
$X\in SmProj_k$ of dimension $d$, and $\alpha \in CH^n(X)_{\mathbb Q}$,
$2\leq n \leq d$.  Then,
\begin{enumerate}
\item  \label{main.thm2.a}
$\alpha \in F^2CH^n(X)_{\mathbb Q}$ 
\eqref{def.orth.fil.Chow} if and only if
\eqref{vanishing.pair1} and \eqref{vanishing.pair2} hold.

\item \label{main.thm2.b}
assume further that $\alpha \in 
CH^n_{\mathrm{alg}}(X)_{\mathbb Q}$, i.e. it is
algebraically equivalent to zero.
Then, $\alpha \in F^2CH^n(X)_{\mathbb Q}$ 
\eqref{def.orth.fil.Chow} if and only if
$\alpha \in CH^n_{\mathrm{inc}}(X)_{\mathbb Q}$ \eqref{def.inc.equiv}.
\end{enumerate}
\end{thm}
\begin{proof}
\eqref{main.thm2.a}:
Assume that
$\alpha \in F^2CH^n(X)_{\mathbb Q}$, then the result follows
from \eqref{prop.pf.nec.cond}.
On the other hand, if \eqref{vanishing.pair1} and \eqref{vanishing.pair2} 
hold, by \eqref{lem.perf.incequiv} and \eqref{lem.fieldext.comp},
it suffices to prove the result when $k$ is algebraically closed.
Then,
we conclude, by \eqref{prop.0-1hom.vanpair1-2}
and \eqref{prop.alpha1ff}, that $\alpha \in F^2CH^n(X)_{\mathbb Q}$.

\eqref{main.thm2.b}:
By \eqref{lem.algeq.vanpair2}, we deduce that
$\alpha \in CH^n_{\mathrm{inc}}(X)_{\mathbb Q}$ \eqref{def.inc.equiv},
if and only if $\alpha$ satisfies \eqref{vanishing.pair1}
and \eqref{vanishing.pair2}.  Then, the
result follows from \ref{main.thm2}\eqref{main.thm2.a}.
\end{proof}

\section{First reductions}  \label{sec.firstres}

\subsection{}  \label{s.pf.main}
With the notation and conditions
of \eqref{def.orth.fil.Chow}-\eqref{sss.main.not}.
Let $Y\in SmProj_k$, $a=-2$ (resp. $Y\in Sm_k$, $a=-1$) and
$\beta \in \Hom_{\DM{k}}(M(Y)(-1)[a], M(X)(-n)[-2n])$ an
arbitrary map.

Consider the following commutative diagram 
in $DM_k$ \eqref{sss.eff.cov}:
\begin{align}  \label{diagram.criteria.thm}
\begin{split}
   \xymatrix{f_{-1}(M(X)(-n)[-2n]) \ar[rr]^-{\epsilon _{-1}^{M(X)(-n)[-2n]}}&&
   M(X)(-n)[-2n] \ar[rr]^-{\alpha(-n)[-2n]}&& \sphere _R\\
   &&M(Y)(-1)[a] \ar[u]_-{\beta} \ar@{-->}[ull]^-{\beta '}&&}
\end{split}
\end{align}
where the existence of the map $\beta '$ follows from 
the universal property of
the counit $\epsilon _{-1}:f_{-1}\rightarrow id$
\cite{MR3614974}*{3.3.1}:
\begin{align}  \label{diagram.univprop}
\begin{split}
\xymatrix{
\Hom _{\DM{k}}(M(Y)(-1)[a], f_{-1}(M(X)(-n)[-2n]))
\ar[d]_{\cong}^{(\epsilon _{-1}^{M(X)(-n)[-2n]})_{\ast}}\\
\Hom _{\DM{k}}(M(Y)(-1)[a], M(X)(-n)[-2n])}
\end{split}
\end{align}

\subsubsection{}  \label{sss.appl.Lieblem}
With the notation and conditions of \eqref{s.pf.main}.
By Lieberman's lemma \eqref{prop.Lieb-lem}
$\alpha \in CH^n(X)_R$ satisfies \eqref{vanishing.pair1} (resp.
\eqref{vanishing.pair2}) if and only if
the map induced by $\alpha (-n)[-2n]$ in
\eqref{diagram.criteria.thm} is zero
for every $Y\in SmProj_k$ and $a=-2$ 
(resp. every $Y\in Sm_k$ and $a=-1$):
\begin{align*}
\xymatrix{
\Hom _{\DM{k}}(M(Y)(-1)[a], M(X)(-n)[-2n])
\ar[d]^{\alpha (-n)[-2n]_{\ast} =0}\\
\Hom _{\DM{k}}(M(Y)(-1)[a], \sphere _R).}
\end{align*}

\begin{prop}  \label{prop.vanpair-vancomp}
Let $k$ be an arbitrary base field, and with the notation and conditions
of \eqref{def.orth.fil.Chow}-\eqref{sss.main.not}.  Then,
$\alpha \in CH^n(X)_R$ satisfies \eqref{vanishing.pair1} (resp.
\eqref{vanishing.pair2}) if and only if
the map induced by the top row
of \eqref{diagram.criteria.thm} is zero
for every $Y\in SmProj_k$ and $a=-2$ 
(resp. every $Y\in Sm_k$ and $a=-1$):
\begin{align*}
\xymatrix{
\Hom _{\DM{k}}(M(Y)(-1)[a], f_{-1}(M(X)(-n)[-2n]))
\ar[d]^{(\alpha (-n)[-2n]\circ
\epsilon _{-1}^{M(X)(-n)[-2n]})_{\ast} =0}\\
\Hom _{\DM{k}}(M(Y)(-1)[a], \sphere _R)}
\end{align*}
\end{prop}
\begin{proof}
The result follows by \eqref{sss.appl.Lieblem} and
the isomorphism \eqref{diagram.univprop}.
\end{proof}

\begin{prop}  \label{prop.pf.nec.cond}
Let $k$ be an arbitrary base field, and with the notation and conditions
of \eqref{def.orth.fil.Chow}-\eqref{sss.main.not}.  Assume further that
$\alpha \in F^2 CH^n(X)_R$ \eqref{def.orth.fil.Chow}.  Then, \eqref{vanishing.pair1} and
\eqref{vanishing.pair2} hold.
\end{prop}
\begin{proof}
Since $\alpha \in F^2 CH^n(X)_R$,
by \cite{MR3614974}*{5.3.2} we conclude that the composition
in the top row of \eqref{diagram.criteria.thm} is zero:
$\alpha(-n)[-2n] \circ \epsilon _{-1}^{M(X)(-n)[-2n]} =0$, so
 the result follows from
\eqref{prop.vanpair-vancomp}.
\end{proof}

\subsection{Zero cycles}  \label{subsec.zerocyc}
With the notation and conditions of 
\eqref{def.orth.fil.Chow}-\eqref{sss.main.not}.
In this section, we assume further that the base field $k$ is algebraically
closed, $R=\mathbb Q$ \eqref{coeffs.not}, and
$\alpha \in CH^d(X)_R$
 satisfies
\eqref{vanishing.pair1}, or equivalently, by \ref{rmk.car.vanpair1}\eqref{rmk.car.vanpair1.c}:
$\alpha \in T(X)_R$, the Albanese kernel of $X$.  We will write
$\inthomeff$ for the internal Hom-functor in $\DMeff{k}$.
Recall that $d$ is the dimension of $X$.

\begin{prop}  \label{prop.zercyd4}
With the notation and conditions of \eqref{subsec.zerocyc}.
Assume further
that $d=\mathrm{dim}\; X \geq 4$.
Then there exists a smooth hyperplane section
$i:H\rightarrow X$ such that the following conditions hold:
\begin{enumerate}
\item There exists $\alpha _H \in CH^{d-1}(H)_{\mathbb Q}$
such that $i_{\ast}(\alpha _H)=\alpha \in CH^d (X)_{\mathbb Q}$, and
\item $\alpha _H \in T(H)_{\mathbb Q}\subseteq CH^{d-1}(H)_{\mathbb Q}$,
the Albanese kernel of $H$.
\end{enumerate}
\end{prop}
\begin{proof}
Combining 
\cite{MR529493}*{Thm. 7}, \cite{MR2620980}
 and \cite{MR2171939}*{XI, Thm. 3.18},
we conclude that there exist a smooth hyperplane section 
$i:H\rightarrow X$ and $\alpha _H \in CH^{d-1}(H)_{\mathbb Q}$
such that $i_{\ast}(\alpha _H)=\alpha$ and
$i^{\ast}:CH^1(X)\stackrel{\cong}{\rightarrow} CH^1(H)$ is an isomorphism.

We observe that the degree of $\alpha \in CH^d (X)_{\mathbb Q}$ 
is zero \ref{rmk.car.vanpair1}\eqref{rmk.car.vanpair1.a}, so
the degree of $\alpha _H \in CH^{d-1}(H)_{\mathbb Q}$ is also zero.
Hence, it is enough to show that the induced map
$Alb(i):Alb(H)\rightarrow Alb(X)$ between the Albanese varieties
is an isogeny.

We fix a closed point $x_0 \in H(k)$, and consider the following 
commutative diagram: 
\begin{align*}
\xymatrix{H \ar[r]^-i \ar[d]_-{alb_H}& X \ar[d]^-{alb_X} \\
	Alb(H) \ar[r]_-{Alb(i)} & Alb(X) }
\end{align*}
where $alb_H$, $alb_X$ are the canonical maps
into the Albanese varieties such that $alb_H(x_0)=0$ and
$alb_X(x_0)=0$.  Now, consider the induced map on the dual
abelian varieties:
\begin{align*}
Alb(i)^t: \widehat{Alb(X)}\cong Pic^0(X) \rightarrow\widehat{Alb(H)}
\cong Pic^0(H)
\end{align*}
where $Pic^0 X$, $Pic^0 H$ are the Picard varieties of
$X$, $H$, respectively.  By \cite{MR106225}*{p.152}, if 
$x\in Pic ^0{X}(k)$ classifies $\mathcal L_x \in CH^1(X)$, then
 $Alb(i)^t(x)\in Pic ^0(H)(k)$ classifies $i^\ast (\mathcal L _x) \in
 CH^1(Y)$.  So, we deduce that $Alb(i)^t$ is an isogeny since
 $i^{\ast}:CH^1(X)\stackrel{\cong}{\rightarrow} CH^1(H)$ 
 is an isomorphism.
 
 Thus, by \cite{MR106225}*{p. 125, Prop. 2} we conclude
 that $Alb(i)$ is an isogeny, which finishes the proof.
\end{proof}

\begin{prop}  \label{prop.3fold}
With the notation and conditions of \eqref{subsec.zerocyc}.
Assume further
that $d=\mathrm{dim}\; X =3$ and that the base field
$k$ is not the algebraic closure of a finite field.
Then there exists a smooth hyperplane section
$i:H\rightarrow X$ such that the following conditions hold:
\begin{enumerate}
\item There exists $\alpha _H \in CH^{2}(H)_{\mathbb Q}$
such that $i_{\ast}(\alpha _H)=\alpha \in CH^3 (X)_{\mathbb Q}$, and
\item $\alpha _H \in T(H)_{\mathbb Q}\subseteq CH^{2}(H)_{\mathbb Q}$,
the Albanese kernel of $H$.
\end{enumerate}
\end{prop}
\begin{proof}
Combining
\cite{MR529493}*{Thm. 7}, \cite{MR2620980}
 with \cite{MR4739668}*{Thm. 1.1},
we conclude that  there exist a smooth hyperplane section 
$i:H\rightarrow X$ and $\alpha _H \in CH^{2}(H)_{\mathbb Q}$
such that $i_{\ast}(\alpha _H)=\alpha$ and
$i^{\ast}:CH^1(X)_{\mathbb Q}
\stackrel{\cong}{\rightarrow} CH^1(H)_{\mathbb Q}$ is an isomorphism.
Then we conclude by an argument parallel
to the proof in \eqref{prop.zercyd4}.
\end{proof}

\begin{prop}  \label{prop.surface}
With the notation and conditions of 
\eqref{subsec.zerocyc}. 
Let $X\in SmProj _k$ be a surface.
If $\alpha \in CH^2(X)_{\mathbb Q}$ satisfies \eqref{vanishing.pair1}, then
 $\alpha \in F^2 CH^2(X)_{\mathbb Q}$.
\end{prop}
\begin{proof}
By \ref{rmk.car.vanpair1}\eqref{rmk.car.vanpair1.c}, we
observe that $\alpha \in T(X)_{\mathbb Q}$, 
the Albanese kernel of $X$.
Now, \cite{MR1061525}*{Thm. 3} and \cite{MR2187153}*{Prop. 14.2.3}
imply that
$M(X)_{\mathbb Q}$ splits as a direct sum in $\DMeff{k}$:
\begin{align*}
M(X)_{\mathbb Q} \cong M_0 (X)\oplus M_1 (X)\oplus
M_2^\mathrm{alg}(X) \oplus t_2(X) \oplus M_3 (X) \oplus M_4(X).
\end{align*}
and it also follows from \cite{MR2187153}*{Prop. 14.2.3} that
$\alpha _{\mathbb Q}: M(X)_{\mathbb Q} 
\rightarrow \sphere _{\mathbb Q} (2)[4]$ factors as:
\begin{align*}
\xymatrix{M(X)_{\mathbb Q} 
\ar[r]^-{\alpha _{\mathbb Q}} \ar[d]_-{\pi}& \sphere _{\mathbb Q} (2)[4]\\
t_2 (X) \ar[ur]_-{\alpha _{AJ}}&}
\end{align*}
where $\pi$ is the projection induced by the splitting of 
$M(X)_{\mathbb Q}$.

Thus,
combining \cite{MR3614974}*{5.3.2} (see
\ref{diagram.criteria.thm})
and \eqref{prop.vanchancoef}-\eqref{prop.slicechangcof}, 
we deduce that
it suffices to show that $f_{-1}(t_2(X)(-2)[-4])
\cong 0$ in $\DM{k}$.
Now, by \cite{MR3614974}*{3.3.3.(2)}:
\begin{align*}
f_{-1}(t_2(X)(-2)[-4]) \cong (f_1(t_2(X)))(-2)[-4].  
\end{align*}
So,
it is enough to show that $(f_1(t_2(X)[2]))(-1)[-2]\cong 0$ in $\DMeff{k}$,
since
the functor $DM_k\rightarrow DM_k$,
$E\mapsto E(1)[4]$ is triangulated and an equivalence of categories.

On the other hand, by \cite{MR2187153}*{Thm. 14.8.4(b)} we
observe that
\begin{align}  \label{eq.vanalbker}
\inthomeff (\sphere _{\mathbb Q}(1), t_2(X))\cong 0 .
\end{align}
Since
$t_2(X)\cong t_2(X)_{\mathbb Q}$ in 
$\DMeff{k}$, by adjointness and
\eqref{prop.tenscoef}-\eqref{prop.vanchancoef} we deduce that
\eqref{eq.vanalbker} implies that:
$\inthomeff (\sphere (1), t_2(X))\cong 0$ in $\DMeff{k}$,
as well. 

Hence, 
the result follows from
 \cite{MR2600283}*{Lem. 5.9}, \cite{MR2249535}*{Prop. 1.1}:
\begin{align*}
(f_1(t_2(X)[2]))(-1)[-2]
\cong \inthomeff (\sphere (1)[2], t_2(X)[2])
\cong \inthomeff (\sphere (1), t_2(X))
\cong 0.
\end{align*}
\end{proof}

\section{Further reductions}  \label{sec.furthred}

\subsection{}  \label{ss.main.suff}
With the notation and conditions
of \eqref{def.orth.fil.Chow}-\eqref{sss.main.not}.
We will consider
Voevodsky's homotopy $t$-structure
$((\DMeff{k})_{\geq 0}, (\DMeff{k})_{\leq 0})$
in $\DMeff{k}$ \cite{MR1764202}*{p. 11}.  
We will follow the homological notation for $t$-structures
\cite{MR2438151}*{\S 2.1.3},
\cite{MR2735752}*{\S 1.3}, and
write $\tau _{\geq m}$, $\tau _{\leq m}$ for the
truncation functors and
$\mathbf{h}_m=[-m](\tau _{\leq m}\circ\tau_{\geq m})$.
Let $\HINST{k}$ denote the abelian category of
homotopy invariant Nisnevich sheaves with transfers
on $Sm_k$, which
is the heart of the homotopy $t$-structure in $\DMeff{k}$.
Given a map $f$ in $\HINST{k}$, we will write
$Ker (f), Coker(f) \in \HINST{k}$ for the kernel of $f$
and the cokernel of $f$, respectively.

\subsubsection{}  \label{sss.tensprod}
We will only consider tensor products in $\DMeff{k}$.

\subsubsection{}  \label{not.inthom}
To simplify the notation, we will write
$\varphi _s:\DMeff{k}\rightarrow \DMeff{k}$, $s\geq 0$ for the
triangulated functor $E\mapsto \inthomeff(\sphere (s)[2s], E)$,
$E\in \DMeff{k}$ \eqref{subsec.zerocyc}.

\begin{prop}  \label{prop.Chow-inthom}
Let $W\in SmProj_k$.  Then, for
every $Y\in Sm_k$, $r\in \mathbb Z$
there is a natural isomorphism \eqref{not.inthom}:
\begin{align*}
\Hom _{\DMeff{k}}(M(Y)[r], \varphi _{s}(M(W)))
 \cong CH^{d-s}(W\times Y,r), s\geq 0.
\end{align*}
\end{prop}
\begin{proof}
In effect, this follows
by adjointness and combining
 Poincar\'e duality 
\cite{MR1764202}*{Thm. 4.3.7},
\cite{MR2399083}*{Prop. 6.7.1 and \S 6.7.3} with
\cite{MR1883180}. 
\end{proof}

\subsubsection{}   \label{sss.Chow-inthom}
By \eqref{prop.Chow-inthom},
we deduce that $\varphi _s(M(W))\in (\DMeff{k})_{\geq 0}$,
$W\in SmProj_k$, $s\geq 0$.  
Then, we obtain the following distinguished triangle in
$\DMeff{k}$:
\begin{align*}
\xymatrix{\tau _{\geq 1} \varphi _{s}(M(W)) \ar[r]^-{t_1}&
\tau _{\geq 0} \varphi _{s}(M(W)) \cong \varphi _{s}(M(W)) 
\ar[r]^-{\sigma _0}&
\mathbf{h}_0 \varphi _{s}(M(W))}
\end{align*}

\begin{rmk}   \label{rmk.Chow-inthom}
It follows from \cite{MR2249535}*{Rmk. 2.3} and \eqref{prop.Chow-inthom}
that the Nisnevich sheaf with transfers $\mathbf{h}_0 \varphi _{s}(M(W))$
is birational, and by the localization sequence for the Chow groups
we deduce that the map induced by $\sigma _0$ is surjective
for every $Y\in Sm_k$:
\begin{align*}
\xymatrix{\Hom _{\DMeff{k}}(M(Y), \varphi _{s}(M(W))) 
\ar[r]^-{\sigma _{0\ast}}&
\Hom _{\DMeff{k}}(M(Y),\mathbf{h}_0 \varphi _{s}(M(W)))\ar[r] &0}
\end{align*}
We observe that $\sigma _{0\ast}$
is the canonical map from the presheaf:
\begin{align*}
Y\in Sm_k \mapsto \Hom _{\DMeff{k}}
(M(Y),\varphi _{s}(M(W)))=CH^{d-s}(W\times Y)
\end{align*}
to its associated Nisnevich sheaf
$\mathbf{h}_0 \varphi _{s}(M(W)):Y\in Sm_k \mapsto
CH^{d-s}(W_{k(Y)})$.
\end{rmk}

\subsubsection{}  \label{sss.inthom.effcov}
Combining 
 \cite{MR2600283}*{Lem. 5.9}, \cite{MR2249535}*{Prop. 1.1}
with \cite{MR3614974}*{3.3.3.(2)} 
we conclude that for $n\geq 1$:
\begin{align*}
   f_{-1}(M(X)(-n)[-2n]) \cong f_{n-1}(M(X))(-n)[-2n]\\
   \cong \inthomeff(\sphere (n-1)[2n-2], M(X))(-1)[-2]
   &=\varphi _{n-1}(M(X)) (-1)[-2]
\end{align*}

\subsubsection{}  \label{sss.DMeff.van.equiv}
Consider the diagram \eqref{diagram.criteria.thm}.  To simplify
the notation,
let 
\begin{align*}
\alpha ^{(1)}
=(\alpha (-n)[-2n]\circ \epsilon _{-1}^{M(X)(-n)[-2n]})(1)[2]:\varphi _{n-1}
(M(X))
\rightarrow  \sphere _R (1)[2],
\end{align*}
which is a map in $\DMeff{k}$
\eqref{subsubsec.cancel}, and
consider the following  diagram in $\DMeff{k}$:
\begin{align}  \label{diagram.1hom.vanpair2}
\begin{split}
   \xymatrix@C=1.2pc{&& (\mathbf{h}_1 \varphi _{n-1}(M(X)))[1] \ar@{-->}  `l[lldd]_-{\alpha _1 ^{(1)}} 
   `[lddd] [lddd]& 
   \mathbf{h}_0 \varphi _{n-1}(M(X)) \ar@{-->} `[ddr] `[ddd] [dddll]^-{\alpha _0 ^{(1)}}&\\
   &\tau _{\geq 2}\varphi _{n-1}(M(X)) \ar[r]^-{t_2} & 
   \tau _{\geq 1}\varphi _{n-1}(M(X))
   \ar[r]^-{t_1} \ar[u]^-{\sigma _1}& \tau _{\geq 0} \varphi _{n-1}(M(X))  \ar[u]^-{\sigma _0} \ar[ddll]^-{\alpha ^{(1)}} \ar@{=}[d]&\\
   &&&\varphi _{n-1}(M(X))&\\
   &\sphere _R (1)[2]&&& }
\end{split}
\end{align}
where $\tau _{\geq i+1}\varphi _{n-1}(M(X)) 
\stackrel{t_{i+1}}{\longrightarrow} 
\tau _{\geq i} \varphi _{n-1}(M(X))
\stackrel{\sigma _i}{\longrightarrow} (\mathbf{h}_i \varphi _{n-1}(M(X)))[i]$ are
distinguished triangles in $\DMeff{k}$ for
$i=0$, $1$ and the isomorphism $\tau _{\geq 0} \varphi _{n-1}(M(X))
\cong  \varphi _{n-1}(M(X))$ follows from
\eqref{sss.Chow-inthom}.

\begin{prop}  \label{prop.1hom.vanpair2}
Let $k$ be an arbitrary base field, and with the notation and conditions
of \eqref{def.orth.fil.Chow}-\eqref{sss.main.not}, \eqref{ss.main.suff},
\eqref{not.inthom},
\eqref{sss.inthom.effcov}-\eqref{diagram.1hom.vanpair2}.  Assume
further that $R$ is flat over $\mathbb Z$ \eqref{coeffs.not}.   Then, 
\begin{enumerate}
\item \label{prop.1hom.vanpair2.a}
there exists a unique map $\alpha _1 ^{(1)}: (\mathbf{h}_1
\varphi _{n-1}(M(X)) )[1]
\rightarrow \sphere _R (1)[2]$ in $\DMeff{k}$ such that
$\alpha _1 ^{(1)}\circ \sigma _1 =\alpha ^{(1)}\circ t_1$
in \eqref{diagram.1hom.vanpair2}.
\item  \label{prop.1hom.vanpair2.b}
assume further that $\alpha \in CH^n(X)_R$
satisfies \eqref{vanishing.pair2}.  Then, the map
$\alpha _1 ^{(1)}=0$ in
\ref{prop.1hom.vanpair2}\eqref{prop.1hom.vanpair2.a} and
there exists a unique map $\alpha _0 ^{(1)}: \mathbf{h}_0
\varphi _{n-1}(M(X)) \rightarrow \sphere _R (1)[2]$ in $\DMeff{k}$ such that
$\alpha _0 ^{(1)}\circ \sigma _0 =\alpha ^{(1)}$
in \eqref{diagram.1hom.vanpair2}.
\end{enumerate}
\end{prop}
\begin{proof}
\eqref{prop.1hom.vanpair2.a}:  We observe that
$\sphere _R (1)[2] \cong (\mathcal O ^\ast  \otimes R) [1]$
by \cite{MR1764202}*{Thm. 3.4.2}, \cite{MR2242284}*{4.1},
\cite{MR2399083}*{3.2}.  Then,
since $\mathcal O^{\ast}$ is in the heart
of the homotopy $t$-structure and $R$ is a flat over $\mathbb Z$,
 we deduce that
$\mathcal O ^\ast  \otimes R$ (see \ref{sss.tensprod}) 
is also in the heart of the
homotopy $t$-structure.  Hence,
$\sphere _R (1)[2] \in (\DMeff{k})_{\leq 1}$ which implies:
\begin{align*}
0 &=\Hom _{\DMeff{k}}(\tau _{\geq 2}\varphi _{n-1}(M(X)) , \sphere _R (1)[2])
 \\
&=\Hom _{\DMeff{k}}((\tau _{\geq 2}\varphi _{n-1}(M(X)) )[1] , \sphere _R (1)[2])
\end{align*}
Then, the result follows since
\begin{align*}
\tau _{\geq 2}\varphi _{n-1}(M(X)) \stackrel{t_{2}}{\longrightarrow} 
\tau _{\geq 1} \varphi _{n-1}(M(X))
\stackrel{\sigma _1}{\longrightarrow} (\mathbf{h}_1 \varphi _{n-1}(M(X)) )[1]
\end{align*}
is a distinguished triangle in $\DMeff{k}$.

\eqref{prop.1hom.vanpair2.b}:
First, we show that $\alpha _1 ^{(1)} =0$ \eqref{diagram.1hom.vanpair2}.  Since
$\alpha _1 ^{(1)} [-1]:\mathbf{h}_1 \varphi _{n-1}(M(X))
 \rightarrow \sphere _R (1)[1]
\cong \mathcal O ^\ast \otimes R$ is a map in $\HINST{k}$ 
\eqref{ss.main.suff}, it suffices to see that for every $Y\in Sm_k$,
the map induced by $\alpha _1 ^{(1)}$ is zero:
\begin{align*}
\xymatrix{\Gamma (Y, \mathbf{h}_1 \varphi _{n-1}(M(X)))\cong
\Hom _{\DMeff{k}}(M(Y)[1],(\mathbf{h}_1 \varphi _{n-1}(M(X)))[1]) 
\ar[d]^-{\alpha _{1 \ast}^{(1)}}\\
 \Gamma (Y, \mathcal O ^\ast \otimes R) \cong
\Hom _{\DMeff{k}}(M(Y)[1],\sphere _R (1)[2])}
\end{align*}
We notice that
\begin{align*}
\xymatrix{ \Hom _{\DMeff{k}} (M(Y)[1],
\tau _{\geq 1}\varphi _{n-1}(M(X)))
\ar[d]^-{\sigma _{1\ast}} \\
\Hom _{\DMeff{k}}(M(Y)[1], (\mathbf{h}_1 \varphi _{n-1}(M(X)))[1])}
\end{align*}
is the canonical map from the presheaf
\begin{align*}
Y\in Sm_k \mapsto \Hom _{\DMeff{k}}
(M(Y)[1],\tau _{\geq 1}\varphi _{n-1}(M(X))),
\end{align*}
to its associated Nisnevich sheaf
$\mathbf{h}_1 \varphi _{n-1}(M(X))$.
Hence, we conclude that it suffices to show that for
every $Y\in Sm_k$,
the map
induced by $\alpha _{1}^{(1)} \circ \sigma _1 =
\alpha ^{(1)} \circ t_1$ \eqref{diagram.1hom.vanpair2}
is zero:
\begin{align*}
\xymatrix{
\Hom _{\DMeff{k}}(M(Y)[1],\tau _{\geq 1} \varphi _{n-1}(M(X))) 
\ar[d]^-{(\alpha _{1}^{(1)} \circ \sigma _{1})_{\ast}=
(\alpha ^{(1)} \circ t_{1})_{\ast}}\\
\Hom _{\DMeff{k}}(M(Y)[1],\sphere _R (1)[2])}
\end{align*}
But this follows from \eqref{prop.vanpair-vancomp} and
the definition of $\alpha ^{(1)}$
\eqref{sss.DMeff.van.equiv}; since
the functor $DM_k\rightarrow DM_k$,
$E\mapsto E(1)[2]$ is triangulated and an equivalence of categories.

Therefore,  $\alpha ^{(1)} \circ t_{1}=
\alpha _{1}^{(1)} \circ \sigma _{1}=0$
which implies the
existence of $\alpha _0 ^{(1)}$ since
$\tau _{\geq 1}\varphi _{n-1}(M(X)) \stackrel{t_{1}}{\longrightarrow} 
\tau _{\geq 0} \varphi _{n-1}(M(X))
\stackrel{\sigma _0}{\longrightarrow} \mathbf{h}_0 
\varphi _{n-1}(M(X)) $ is a 
distinguished triangle in $\DMeff{k}$.  To show the
uniqueness of $\alpha _0 ^{(1)}$, it suffices to
see that 
\begin{align*}
\Hom _{\DMeff{k}}((\tau _{\geq 1}\varphi _{n-1}(M(X)))[1] , \sphere _R (1)[2])
=0
\end{align*}
which holds since we have already seen that
$\sphere _R (1)[2] \in (\DMeff{k})_{\leq 1}$.
\end{proof}

\begin{prop}  \label{prop.0-1hom.vanpair1-2}
With the notation and conditions
of \eqref{prop.1hom.vanpair2}.   
Assume further that $\alpha \in CH^n(X)_R$ satisfies
\eqref{vanishing.pair1} and \eqref{vanishing.pair2}.
Then, the following conditions are equivalent:
\begin{enumerate}
\item \label{prop.0-1hom.A}
The map $\alpha _0 ^{(1)} =0$ in 
\ref{prop.1hom.vanpair2}\eqref{prop.1hom.vanpair2.b}.
\item \label{prop.0-1hom.B}
$\alpha \in F^2CH^n(X)_R$.
\end{enumerate}
\end{prop}
\begin{proof}
\eqref{prop.0-1hom.A} $\Rightarrow$ \eqref{prop.0-1hom.B}:  
By \ref{prop.1hom.vanpair2}\eqref{prop.1hom.vanpair2.b},
we conclude that $0=\alpha ^{(1)}:\varphi _{n-1}(M(X))
\rightarrow \sphere _R (1)[2]$ \eqref{diagram.1hom.vanpair2}.
So, by definition of $\alpha ^{(1)}$ \eqref{sss.DMeff.van.equiv}:
\begin{align*}
0=\alpha ^{(1)}(-1)[-2]
=\alpha (-n)[-2n]\circ \epsilon _{-1}^{M(X)(-n)[-2n]},
\end{align*}
which is
the composition
in the top row of \eqref{diagram.criteria.thm}.
Then, the result follows from 
\cite{MR3614974}*{5.3.2}.

\eqref{prop.0-1hom.B} $\Rightarrow$ \eqref{prop.0-1hom.A}:
By the uniqueness in  \ref{prop.1hom.vanpair2}\eqref{prop.1hom.vanpair2.b},
it suffices to show that $\alpha ^{(1)}=0$  
\eqref{sss.DMeff.van.equiv}-\eqref{diagram.1hom.vanpair2}.
Now, we observe that
$\alpha \in F^2CH^n(X)_R$, so \cite{MR3614974}*{5.3.2}
implies that: 
$0=\alpha (-n)[-2n]\circ \epsilon _{-1}^{M(X)(-n)[-2n]}$,
which is the composition in the top row of \eqref{diagram.criteria.thm}.
Then, by the
definition of $\alpha ^{(1)}$ \eqref{sss.DMeff.van.equiv}:
\begin{align*}
0 =(\alpha (-n)[-2n]\circ \epsilon _{-1}^{M(X)(-n)[-2n]})(1)[2]
=\alpha ^{(1)},
\end{align*}
which finishes the proof.
\end{proof}

\begin{rmk}  \label{rmk.AbelJacobi}
With the notation and conditions
of \eqref{prop.1hom.vanpair2}.  Let $CH^n_\mathrm{H}(X)_R$
be the $R$-submodule of $CH^n(X)_R$
where
\eqref{vanishing.pair2} holds.  
Recall that $\mathbf{h}_0 
\varphi _{n-1}(M(X))$ \eqref{rmk.Chow-inthom}
is the Nisnevich sheaf with transfers $Y\in Sm_k \mapsto 
CH^{d-n+1}(X_{k(Y)})$.  So, we will write
$\mathcal{CH}^{d-n+1}(X)$  for $\mathbf{h}_0 
\varphi _{n-1}(M(X))$.
Then, combining
\eqref{prop.1hom.vanpair2} and \eqref{prop.0-1hom.vanpair1-2}
we obtain a short exact sequence:
\begin{align*}
\xymatrix@C=1.2pc@R=0.1pc{0 \ar[r]& F^2CH^n(X)_R \ar[r]& 
CH^n_\mathrm{H}(X)_R \ar[r]&
\Hom _{\DMeff{k}}(\mathcal{CH}^{d-n+1}(X), \mathcal O ^\ast \otimes 
R [1])\\
&& \alpha \ar@{|->}[r]& \alpha_0 ^{(1)}}
\end{align*}
which is natural in $X$ with respect to Chow correspondences.

The reader may compare the map $\alpha_0 ^{(1)}$ in
\eqref{diagram.1hom.vanpair2}-\eqref{prop.1hom.vanpair2},
with the extension constructed by Bloch in \cite{MR991974}*{(2.1) and
Prop. 3}.
\end{rmk}

\subsection{}  \label{ss.torsioncoef}
With the notation and conditions
of \eqref{def.orth.fil.Chow}-\eqref{sss.main.not}.
In this section, we assume further that the base field $k$ is perfect
of exponential characteristic $p$.

\subsubsection{}  \label{step1}
Let $Y\in SmProj _k$.  Then, $M(Y)\in (\DMeff{k})_{\geq 0}$ 
\eqref{sss.Chow-inthom}.
Thus, given
$\beta \in
\Hom _{\DMeff{k}}(M(Y), \varphi _{n-1}(M(X)))
 \cong CH^{d-n+1}(X\times Y)$ \eqref{prop.Chow-inthom},
there exists a unique map $\beta _0$ making the following
diagram in $\DMeff{k}$
commute:
\begin{align}  \label{step1.diag}
\begin{split}
\xymatrix{M(Y)\ar[d]_-{\beta} \ar[r]^-{\sigma _Y}
& \mathbf{h}_0 M(Y) \ar[d]^-{\beta _0}\\
\tau _{\geq 0}\varphi _{n-1}(M(X)) \cong \varphi _{n-1}(M(X)) \ar[r]^-{\sigma _0} & \mathbf{h}_0 \varphi _{n-1}(M(X)) }
\end{split}
\end{align}
where $\sigma _0$ is the map in \eqref{diagram.1hom.vanpair2}.

\subsubsection{}  \label{step2}
Let
\begin{align*}
\mathcal P =\bigoplus \limits_{\substack{\beta \in 
	\Hom _{\DMeff{k}}(M(Y), \varphi _{n-1}(M(X)))\\
	Y\in SmProj_k}} \mathbf{h}_0M(Y),
\end{align*}
and consider the map in $\DMeff{k}$ induced by \eqref{step1.diag} on
each direct summand of $\mathcal P$:
\begin{align*}
\xymatrix{ \mathcal P =
\bigoplus \limits_{\substack{\beta \in 
	\Hom _{\DMeff{k}}(M(Y), \varphi _{n-1}(M(X)))\\
	Y\in SmProj_k}}
	 \mathbf{h}_0M(Y) \ar[r]^-{(\beta _0)}& \mathbf{h}_0
	 \varphi _{n-1}(M(X)) }
\end{align*}

\begin{prop}  \label{step2.prop}
With the notation and conditions of \eqref{ss.torsioncoef}.
Then,
the map $(\beta _0)\otimes \mathbb Z [\frac{1}{p}]:
\mathcal P \otimes \mathbb Z [\frac{1}{p}]  
\rightarrow \mathbf{h}_0 \varphi _{n-1}(M(X)) 
\otimes \mathbb Z [\frac{1}{p}]$ (see \ref{sss.tensprod})
is surjective in $\HINST{k}$ \eqref{ss.main.suff}.
\end{prop}
\begin{proof}
To simplify the notation we will omit $\mathbb Z [\frac{1}{p}]$.
By Voevodsky's Gersten's resolution \cite{MR1764200}*{Thm. 4.37} 
it suffices to show
that for every finitely generated field extension $L/k$, the map
induced on stalks is surjective,
$(\beta _{0})_L: \mathcal P _L \rightarrow \mathbf{h}_0 
\varphi _{n-1}(M(X))_L$.  

Now, we observe that
$\mathcal P$ and $\mathbf{h}_0 \varphi _{n-1}(M(X))$ are
birational sheaves \eqref{rmk.Chow-inthom}.
Thus, if $Y\in SmProj_k$ with function field $k(Y)$ we obtain
the following commutative diagram where the vertical
arrows are the canonical maps to the stalks, which
are isomorphisms by birationality:
\begin{align}  \label{step2.diag}
\begin{split}
\xymatrix{\Gamma(Y, \mathcal P) \ar[rr]^-{(\beta _0)(Y)} \ar[d]_-{\cong}&& 
\Gamma (Y, \mathbf{h}_0 \varphi _{n-1}(M(X))) \ar[d]^-{\cong}
\ar[r]& 0 \\
\mathcal P _{k(Y)} \ar[rr]_-{(\beta _0)_{k(Y)}}
&& \mathbf{h}_0 
\varphi _{n-1}(M(X))_{k(Y)}&}
\end{split}
\end{align}
and the sujectivity of the
top horizontal arrow follows by 
\eqref{rmk.Chow-inthom}
 and the construction
of $\mathcal P$, $(\beta _0)$ \eqref{step1.diag}-\eqref{step2}.

Now, if $k$ has characteristic zero, by 
Hironaka's resolution of singularities
\cite{MR199184}*{Cor. p. 132}
there exists $Y \in SmProj_k$ such that $k(Y)=L$, so the
surjectivity of $(\beta _0)_L$
 follows from \eqref{step2.diag}.  
In case $k$ has positive characteristic $p$,
the work of
de Jong, Gabber, Temkin \cite{MR3665001}*{Thm. 1.2.5},
\cite{MR3329779}*{Thm. 2.1}, \cite{MR1423020}*{Thm. 4.1} implies
the existence of $Y\in SmProj_k$
such that $k(Y)/L$ is a finite extension of degree $p^r$, and then
the surjectivity of $(\beta _0)_L$
follows by a transfer argument from \eqref{step2.diag} since 
$p^r$ is a unit in $\mathbb Z [\frac{1}{p}]$. 
\end{proof}

The following lemma will be necessary in the next section:

\begin{lem}  \label{step2.lem}
With the notation and conditions of \eqref{ss.torsioncoef}.
Let $\mathcal F , \mathcal G \in \HINST{k}$ \eqref{ss.main.suff}
be sheaves of $\mathbb Z [\tfrac{1}{p}]$-modules.
Assume that $\mathcal F$ is birational, and that
$\mathcal G$ satisfies the following condition:
$\Gamma (Y, \mathcal G)=0$ for every $Y\in SmProj_k$.
Then,
$\Hom _{\HINST{k}}(\mathcal F  , \mathcal G)=0.$
\end{lem}
\begin{proof}
Let $f\in \Hom _{\HINST{k}}(\mathcal F  , \mathcal G)$.
To conclude  that $f=0$, it suffices to show, by Voevodsky's
Gersten's resolution \cite{MR1764200}*{Thm. 4.37}, that for every
finitely generated field extension $L/k$, the map induced
on stalks is zero, $0=f_L: \mathcal F _L \rightarrow \mathcal G _L$.

Now, for $Y\in SmProj _k$ with function field $k(Y)$, we obtain
the following commutative diagram where the vertical arrows
are the canonical maps to the stalks:
\begin{align}  \label{lem.diag}
\begin{split}
\xymatrix{\Gamma(Y, \mathcal F) \ar[rr]^-{f(Y)} \ar[d]_-{\cong}&& 
\Gamma (Y, \mathcal G ) = 0  \ar[d]\\
\mathcal F _{k(Y)} \ar[rr]_-{f_{k(Y)}}
&& \mathcal G _{k(Y)}}
\end{split}
\end{align}
and the left vertical arrow is an 
isomorphism, since $\mathcal F$ is birational.
Thus, we conclude that $f_{k(Y)}=0$.

Then, 
the result follows from \eqref{lem.diag}, applying
the argument after \eqref{step2.diag}.
\end{proof}

\subsubsection{}  \label{step3}
Let
\begin{align*}
\xymatrix{\mathcal K \ar[r]&\mathcal P \ar[r]^-{(\beta _0)}& 
\mathbf{h}_0 \varphi _{n-1}(M(X))}
\end{align*}
be a distinguished triangle in $\DMeff{k}$.
Then, since $\{ \mathbf{h}_i:\DMeff{k}\rightarrow \mathbf{HI}_k,
 i\in \mathbb Z\}$ is a cohomological functor \cite{MR751966}*{Thm. 1.3.6}, 
 by \eqref{step2.prop}
 we conclude that (see \ref{sss.tensprod}):
\begin{align}  \label{step3.diagram}
\mathcal K \otimes \mathbb Z 
[\tfrac{1}{p}] \cong Ker (\beta _0) \otimes \mathbb Z 
[\tfrac{1}{p}]
 \in \HINST{k}.
\end{align}

\subsubsection{}  \label{step4}
In the rest of this section we assume further that
$\alpha \in CH^n(X)_R$
satisfies \eqref{vanishing.pair1} and \eqref{vanishing.pair2},
and that $R$ is flat over $\mathbb Z$ \eqref{coeffs.not}.  Then,
combining \eqref{step1} 
and \ref{prop.1hom.vanpair2}\eqref{prop.1hom.vanpair2.b},
we obtain the following commutative diagram in $\DMeff{k}$:
\begin{align}  \label{step4.diagram}
\begin{split}
\xymatrix{M(Y)\ar[d]_-{\beta} \ar[r]^-{\sigma _Y}
& \mathbf{h}_0 M(Y) \ar[d]^-{\beta _0}\\
\tau _{\geq 0}\varphi _{n-1}(M(X)) \cong \varphi _{n-1}(M(X)) 
\ar[r]^-{\sigma _0} 
\ar[d]_-{\alpha ^{(1)}}& 
\mathbf{h}_0 \varphi _{n-1}(M(X)) \ar[dl]^-{\alpha _0 ^{(1)}}\\
		\sphere _R (1)[2]&}
\end{split}
\end{align}
for any $\beta \in
\Hom _{\DMeff{k}}(M(Y), \varphi _{n-1}(M(X))) 
\cong CH^{d-n+1}(X\times Y)$ \eqref{prop.Chow-inthom}, and
where $\alpha ^{(1)}$, $\alpha ^{(1)}_0$
are the maps in 
\ref{prop.1hom.vanpair2}\eqref{prop.1hom.vanpair2.b}.

Now, since
the functor $DM_k\rightarrow DM_k$,
$E\mapsto E(1)[2]$ is triangulated and an equivalence of categories,
by \eqref{prop.vanpair-vancomp} and
the definition of $\alpha ^{(1)}$
\eqref{sss.DMeff.van.equiv} we deduce that
the map induced by $\alpha ^{(1)}$  is zero:
\begin{align*}
\xymatrix{
\Hom _{\DMeff{k}}(M(Y), \varphi _{n-1}(M(X))) 
\ar[r]^-{\alpha ^{(1)} _{\ast} =0} &
\Hom _{\DMeff{k}}(M(Y),\sphere _R (1)[2])}
\end{align*}
Thus, in \eqref{step4.diagram}:
\begin{align*}
0=\alpha ^{(1)} \circ \beta =
(\alpha _0 ^{(1)} \circ \beta _0 ) \circ \sigma _Y
\end{align*}
Then, since
\begin{align*}
\xymatrix{\tau _{\geq 1}M(Y) \ar[r] & M(Y) \ar[r]^-{\sigma _Y}
& \mathbf{h}_0 M(Y)}
\end{align*}
is a distinguished triangle in $\DMeff{k}$,
$\Hom _{\DMeff{k}}((\tau _{\geq 1}M(Y))[1], \sphere _R (1)[2])=0$
and $(\alpha _0 ^{(1)}\circ \beta _0 ) \circ \sigma _Y=0$;
we conclude that in \eqref{step4.diagram}:
\begin{align*}
\alpha _0 ^{(1)} \circ  \beta _0  =0.
\end{align*}
Hence, by construction of $\mathcal P$ and $(\beta _0)$ 
 \eqref{step2}, we deduce that:
\begin{align}  \label{step4.map}
\begin{split}
\xymatrix{0=\alpha _0 ^{(1)}\circ (\beta _0):\mathcal P
\rightarrow \sphere _R(1)[2].}
\end{split}
\end{align}

\subsubsection{}  \label{step5}
By \eqref{step3} and \eqref{step4.map}
there exists $\gamma ^X \in
\Hom _{\DMeff{k}}(\mathcal K , \sphere _R (1)[1])$
such that the following diagram in $\DMeff{k}$ commutes:
\begin{align*}
\xymatrix{\mathcal K \ar[r]&\mathcal P \ar[r]^-{(\beta _0)} \ar[dr]_-{0}
& \mathbf{h}_0
\varphi _{n-1}(M(X)) \ar[d]^-{\alpha _0 ^{(1)}} \ar[r]^-{\delta}&\mathcal K [1]
\ar[dl]^-{\gamma ^X [1]}\\
&&\sphere _R (1)[2] &}
\end{align*}
where $\alpha _0 ^{(1)}$ is the map in 
\ref{prop.1hom.vanpair2}\eqref{prop.1hom.vanpair2.b}.

\subsection{}  \label{ss.consuni}
In the rest of this section we  assume further that the base field
$k$ is algebraically closed and that $R=\mathbb Z [\tfrac{1}{p}]$
\eqref{coeffs.not}.

\subsubsection{}  \label{step6}
Let $k^\ast \in \HINST{k}$ \eqref{ss.main.suff}
be the constant sheaf of units in the base 
field \cite{MR2242284}*{2.2}, and consider the canonical map
 in $\HINST{k}$:
\begin{align*}
\xymatrix{0\ar[r] & k^\ast  \ar[r]^-l & \mathcal O ^\ast }
\end{align*}
which is an inclusion, since $k$ is algebraically closed \eqref{ss.consuni},
so $Y(k)\neq \emptyset$ for every $Y\in Sm_k$.
Let
\begin{align*}
\xymatrix{ k^\ast  \ar[r]^-l & \mathcal O ^\ast  \ar[r]^-m&
\mathcal C}
\end{align*}
be a distinguished triangle in $\DMeff{k}$.  Then,
since $\{ \mathbf{h}_i:\DMeff{k}\rightarrow \mathbf{HI}_k,
 i\in \mathbb Z\}$ is a cohomological functor \cite{MR751966}*{Thm. 1.3.6}, 
 we deduce that
$\mathcal C \cong Coker(l) \in \HINST{k}$ \eqref{ss.main.suff}.

\subsubsection{}
Consider the following diagram in $\DMeff{k}$:
\begin{align}  \label{diag.unitfac}
\begin{split}
\xymatrix{& \mathcal K \ar[d]^-{\gamma ^X}
  \ar@{-->}[dl]_{\gamma _1 ^X} &  \\
k^\ast \otimes R \ar[r]^-{l_R} & \mathcal O ^\ast \otimes R \ar[r]^-{m_R}&
\mathcal C \otimes R} 
\end{split}
\end{align}
where $\gamma ^X:\mathcal K \rightarrow \sphere _R (1)[1]\cong 
\mathcal O ^\ast \otimes R$  is the map in
\eqref{step5}.

\begin{prop}  \label{prop.unitfac}
With the notation and conditions
of \eqref{ss.torsioncoef}, \eqref{step4}, \eqref{ss.consuni} and
\eqref{diag.unitfac}.  Then, $m_R \circ \gamma ^X=0$ in
\eqref{diag.unitfac}.
\end{prop}
\begin{proof}
We observe that $R=\mathbb Z [\tfrac{1}{p}]$ \eqref{ss.consuni}.
By \eqref{prop.vanchancoef}, it suffices to show that the
following composition is zero:
\begin{align*}
\xymatrix{\mathcal K \otimes R \ar[r]^-{\gamma ^X \otimes R}& 
(\mathcal O  ^\ast \otimes R)\otimes R \ar[r]^-{m_R\otimes R}&
(\mathcal C \otimes R)\otimes R.}
\end{align*}

Consider the following
distinguished triangle in $\DMeff{k}$ \eqref{step3}:
\begin{align*}
\xymatrix{\mathcal K \ar[r]&\mathcal P \ar[r]^-{(\beta _0)}& 
\mathbf{h}_0 \varphi _{n-1}(M(X))}
\end{align*}
where $\mathcal P$  \eqref{step2}
and  $\mathbf{h} _0 \varphi _{n-1}(M(X))$ are
in the heart of the homotopy $t$-structure, $\HINST{k}$; and
also are birational sheaves \eqref{rmk.Chow-inthom}.
By \eqref{prop.tenscoef}, we deduce
that $\mathcal P\otimes R$, $\mathbf{h} _0 \varphi _{n-1}(M(X))
\otimes R \in \HINST{k}$ (see \ref{sss.tensprod})
and that they are  birational sheaves as well.
On the other hand, $\mathcal K \otimes R\in \HINST{k}$ by
\eqref{step3.diagram}.  Then, \cite{MR3737321}*{Prop. 2.6.2} implies
that $\mathcal K \otimes R$ is a birational sheaf.

Thus, combining \eqref{step2.lem} with \eqref{prop.tenscoef}
we conclude that
it suffices to see that
for every $Y\in SmProj _k$: $\Gamma (Y, \mathcal C )=0$.

Now, since the base field $k$ is algebraically closed, we conclude that
the map induced by $l$ in \eqref{step6} is an isomorphism,
$l_{\ast}: \Gamma (Y, k^\ast) \rightarrow
\Gamma (Y, \mathcal O ^\ast)$, for every $Y\in SmProj_k$.
Since
$k^\ast  \stackrel{l}{\longrightarrow}  \mathcal O ^\ast  
\stackrel{m}{\longrightarrow} \mathcal C$ is a distinguished
triangle in $\DMeff{k}$ \eqref{step6}, and:
\begin{align*}
\Hom _{\DMeff{k}}(M(Y), k^\ast [1])\cong
H^1_{Nis}(Y, k^\ast)\cong H^1_{Zar}(Y, k^\ast)=0, 
\end{align*}
we deduce that
$\Gamma (Y, \mathcal C)=0$ for every $Y\in SmProj_k$,
which finishes the proof.
\end{proof}

\subsubsection{}
We observe that the bottom row in \eqref{diag.unitfac} is
a distinguished triangle in $\DMeff{k}$ \eqref{step6}.  So, by
\eqref{prop.unitfac} there exists $\gamma _1 ^X$ such that
\eqref{diag.unitfac} commutes,
and thus by \eqref{step5}
 the following diagram  in $\DMeff{k}$ commutes:
\begin{align}  \label{diag.fact.pairing}
\begin{split}
\xymatrix{\mathbf{h}_0
\varphi _{n-1}(M(X)) \ar[d]_-{\alpha _0 ^{(1)}} \ar[r]^-{\delta}&\mathcal K [1]
\ar[dl]_-{\gamma ^X [1]} \ar[d]^-{\gamma _1 ^X [1]}\\
\sphere _R (1)[2]\cong \mathcal O ^\ast \otimes R [1] & 
\ar[l]^-{l_R[1]} k^\ast \otimes R [1]}
\end{split}
\end{align}
where $\alpha _0 ^{(1)}$is the map in 
\ref{prop.1hom.vanpair2}\eqref{prop.1hom.vanpair2.b}.

\subsection{Proof of the main result for the algebraic closure of a finite field}

\subsubsection{}  \label{sss.finitefield}
In the rest of this section, we  assume further that the base field
$k$ is the algebraic closure of a finite field and that $R=\mathbb Q$
\eqref{coeffs.not}.

\begin{prop}  \label{prop.vanish.units}
With the notation and conditions
of \eqref{ss.torsioncoef}, \eqref{step4},  \eqref{step6}
 and \eqref{sss.finitefield}.
 Then $k^\ast \otimes \mathbb Q \cong 0$ in $\DMeff{k}$.
\end{prop}
\begin{proof}
We observe that $k^\ast \in \HINST{k}$, so $k^\ast \otimes
\mathbb Q \in \HINST{k}$ by \eqref{prop.tenscoef}.
So it only remains to show that for every $Y\in Sm_k$:
$\Hom _{\DMeff{k}}(M(Y), k^\ast \otimes \mathbb Q)=0$.

Now, by \eqref{prop.tenscoef} we deduce that:
\begin{align*}
\Hom _{\DMeff{k}}(M(Y), k^\ast \otimes \mathbb Q) \cong
\Hom _{\DMeff{k}}(M(Y), k^\ast )\otimes \mathbb Q  \cong
k^\ast \otimes \mathbb Q ,
\end{align*}
and since
$k^\ast$ is a torsion group for the algebraic closure of a finite field,
we conclude that $0\cong k^\ast \otimes \mathbb Q$,
which finishes the proof.
\end{proof}

\begin{prop}  \label{prop.alpha1ff}
With the notation and conditions
of \eqref{ss.torsioncoef}, \eqref{step4},  
 and \eqref{sss.finitefield}.
 Then, $\alpha _0 ^{(1)} =0$ in
  \ref{prop.1hom.vanpair2}\eqref{prop.1hom.vanpair2.b}.
\end{prop}
\begin{proof}
By \eqref{diag.fact.pairing} it suffices to show that the map
$\gamma _1 ^X [1] =0$, which follows from \eqref{prop.vanish.units}.
\end{proof}


\bibliography{biblio_inc-equiv}

\begin{bibdiv}
\begin{biblist}

\bib{MR2438151}{article}{
      author={Ayoub, Joseph},
       title={Les six op\'erations de {G}rothendieck et le formalisme des
  cycles \'evanescents dans le monde motivique. {II}},
        date={2007},
        ISSN={0303-1179},
     journal={Ast\'erisque},
      number={315},
       pages={vi+364 pp. (2008)},
      review={\MR{2438151}},
}

\bib{MR2735752}{article}{
      author={Ayoub, Joseph},
       title={The {$n$}-motivic {$t$}-structures for {$n=0$}, {$1$} and {$2$}},
        date={2011},
        ISSN={0001-8708},
     journal={Adv. Math.},
      volume={226},
      number={1},
       pages={111\ndash 138},
         url={http://dx.doi.org.sci-hub.io/10.1016/j.aim.2010.06.011},
      review={\MR{2735752}},
}

\bib{MR751966}{incollection}{
      author={Be{\u \i}linson, A.~A.},
      author={Bernstein, J.},
      author={Deligne, P.},
       title={Faisceaux pervers},
        date={1982},
   booktitle={Analysis and topology on singular spaces, {I} ({L}uminy, 1981)},
      series={Ast\'erisque},
      volume={100},
   publisher={Soc. Math. France, Paris},
       pages={5\ndash 171},
      review={\MR{751966}},
}

\bib{MR923131}{incollection}{
      author={Be{\u \i}linson, A.~A.},
       title={Height pairing between algebraic cycles},
        date={1987},
   booktitle={{$K$}-theory, arithmetic and geometry ({M}oscow, 1984--1986)},
      series={Lecture Notes in Math.},
      volume={1289},
   publisher={Springer},
     address={Berlin},
       pages={1\ndash 25},
         url={http://dx.doi.org/10.1007/BFb0078364},
      review={\MR{923131 (89h:11027)}},
}

\bib{MR1780498}{article}{
      author={Beke, Tibor},
       title={Sheafifiable homotopy model categories},
        date={2000},
        ISSN={0305-0041},
     journal={Math. Proc. Cambridge Philos. Soc.},
      volume={129},
      number={3},
       pages={447\ndash 475},
         url={http://dx.doi.org/10.1017/S0305004100004722},
      review={\MR{1780498 (2001i:18015)}},
}

\bib{MR2723320}{book}{
      author={Bloch, Spencer},
       title={Lectures on algebraic cycles},
     edition={Second},
      series={New Mathematical Monographs},
   publisher={Cambridge University Press, Cambridge},
        date={2010},
      volume={16},
        ISBN={978-0-521-11842-2},
         url={https://doi.org/10.1017/CBO9780511760693},
      review={\MR{2723320}},
}

\bib{MR2620980}{book}{
      author={Bloch, Spencer~Janney},
       title={Algebraic cohomology classes on algebraic varieties},
   publisher={ProQuest LLC, Ann Arbor, MI},
        date={1971},
  url={http://gateway.proquest.com/openurl?url_ver=Z39.88-2004&rft_val_fmt=info:ofi/fmt:kev:mtx:dissertation&res_dat=xri:pqdiss&rft_dat=xri:pqdiss:7201276},
        note={Thesis (Ph.D.)--Columbia University},
      review={\MR{2620980}},
}

\bib{MR0852815}{article}{
      author={Bloch, Spencer},
       title={Algebraic cycles and higher {$K$}-theory},
        date={1986},
        ISSN={0001-8708},
     journal={Adv. in Math.},
      volume={61},
      number={3},
       pages={267\ndash 304},
         url={https://doi.org/10.1016/0001-8708(86)90081-2},
      review={\MR{852815}},
}

\bib{MR991974}{incollection}{
      author={Bloch, Spencer},
       title={Cycles and biextensions},
        date={1989},
   booktitle={Algebraic {$K$}-theory and algebraic number theory ({H}onolulu,
  {HI}, 1987)},
      series={Contemp. Math.},
      volume={83},
   publisher={Amer. Math. Soc., Providence, RI},
       pages={19\ndash 30},
         url={https://doi.org/10.1090/conm/083/991974},
      review={\MR{991974}},
}

\bib{MR2399083}{article}{
      author={Beilinson, Alexander},
      author={Vologodsky, Vadim},
       title={A {DG} guide to {V}oevodsky's motives},
        date={2008},
        ISSN={1016-443X,1420-8970},
     journal={Geom. Funct. Anal.},
      volume={17},
      number={6},
       pages={1709\ndash 1787},
         url={https://doi.org/10.1007/s00039-007-0644-5},
      review={\MR{2399083}},
}

\bib{MR1423020}{article}{
      author={de~Jong, A.~J.},
       title={Smoothness, semi-stability and alterations},
        date={1996},
        ISSN={0073-8301,1618-1913},
     journal={Inst. Hautes \'Etudes Sci. Publ. Math.},
      number={83},
       pages={51\ndash 93},
         url={http://www.numdam.org/item?id=PMIHES_1996__83__51_0},
      review={\MR{1423020}},
}

\bib{MR4486247}{article}{
      author={Galindo, Cesar},
      author={Pelaez, Pablo},
       title={On the convergence of the orthogonal spectral sequence},
        date={2022},
        ISSN={1532-0073,1532-0081},
     journal={Homology Homotopy Appl.},
      volume={24},
      number={2},
       pages={389\ndash 401},
      review={\MR{4486247}},
}

\bib{MR0233825}{article}{
      author={Griffiths, Phillip~A.},
       title={Periods of integrals on algebraic manifolds. {II}. {L}ocal study
  of the period mapping},
        date={1968},
        ISSN={0002-9327,1080-6377},
     journal={Amer. J. Math.},
      volume={90},
       pages={805\ndash 865},
         url={https://doi.org/10.2307/2373485},
      review={\MR{233825}},
}

\bib{MR0309937}{incollection}{
      author={Griffiths, Phillip~A.},
       title={Some transcendental methods in the study of algebraic cycles},
        date={1971},
   booktitle={Several complex variables, {II} ({P}roc. {I}nternat. {C}onf.,
  {U}niv. {M}aryland, {C}ollege {P}ark, {M}d., 1970)},
      series={Lecture Notes in Math.},
      volume={Vol. 185},
   publisher={Springer, Berlin-New York},
       pages={1\ndash 46},
      review={\MR{309937}},
}

\bib{MR2171939}{book}{
      author={Grothendieck, Alexander},
      editor={Laszlo, Yves},
       title={Cohomologie locale des faisceaux coh\'erents et th\'eor\`emes de
  {L}efschetz locaux et globaux ({SGA} 2)},
      series={Documents Math\'ematiques (Paris) [Mathematical Documents
  (Paris)]},
   publisher={Soci\'et\'e{} Math\'ematique de France, Paris},
        date={2005},
      volume={4},
        ISBN={2-85629-169-4},
        note={S\'eminaire de G\'eom\'etrie Alg\'ebrique du Bois Marie, 1962.,
  Augment\'e{} d'un expos\'e{} de Mich\`ele Raynaud. [With an expos\'e{} by
  Mich\`ele Raynaud], Revised reprint of the 1968 French original},
      review={\MR{2171939}},
}

\bib{MR1944041}{book}{
      author={Hirschhorn, Philip~S.},
       title={Model categories and their localizations},
      series={Mathematical Surveys and Monographs},
   publisher={American Mathematical Society},
     address={Providence, RI},
        date={2003},
      volume={99},
        ISBN={0-8218-3279-4},
      review={\MR{1944041 (2003j:18018)}},
}

\bib{MR199184}{article}{
      author={Hironaka, Heisuke},
       title={Resolution of singularities of an algebraic variety over a field
  of characteristic zero. {I}, {II}},
        date={1964},
        ISSN={0003-486X},
     journal={Ann. of Math. (2)},
      volume={79},
       pages={109\ndash 203; 79 (1964), 205\ndash 326},
         url={https://doi.org/10.2307/1970547},
      review={\MR{199184}},
}

\bib{MR2249535}{article}{
      author={Huber, Annette},
      author={Kahn, Bruno},
       title={The slice filtration and mixed {T}ate motives},
        date={2006},
        ISSN={0010-437X},
     journal={Compos. Math.},
      volume={142},
      number={4},
       pages={907\ndash 936},
         url={http://dx.doi.org.sci-hub.io/10.1112/S0010437X06002107},
      review={\MR{2249535}},
}

\bib{MR1860878}{article}{
      author={Hovey, Mark},
       title={Spectra and symmetric spectra in general model categories},
        date={2001},
        ISSN={0022-4049},
     journal={J. Pure Appl. Algebra},
      volume={165},
      number={1},
       pages={63\ndash 127},
         url={http://dx.doi.org/10.1016/S0022-4049(00)00172-9},
      review={\MR{1860878 (2002j:55006)}},
}

\bib{MR3329779}{incollection}{
      author={Illusie, Luc},
      author={Temkin, Michael},
       title={Expos\'{e} {X}. {G}abber's modification theorem (log smooth
  case)},
        date={2014},
       pages={167\ndash 212},
        note={Travaux de Gabber sur l'uniformisation locale et la cohomologie
  \'{e}tale des sch\'{e}mas quasi-excellents},
      review={\MR{3329779}},
}

\bib{MR1265533}{incollection}{
      author={Jannsen, Uwe},
       title={Motivic sheaves and filtrations on {C}how groups},
        date={1994},
   booktitle={Motives ({S}eattle, {WA}, 1991)},
      series={Proc. Sympos. Pure Math.},
      volume={55},
   publisher={Amer. Math. Soc.},
     address={Providence, RI},
       pages={245\ndash 302},
      review={\MR{1265533 (95c:14006)}},
}

\bib{MR4739668}{article}{
      author={Ji, Lena},
       title={The {N}oether-{L}efschetz theorem in arbitrary characteristic},
        date={2024},
        ISSN={1056-3911,1534-7486},
     journal={J. Algebraic Geom.},
      volume={33},
      number={3},
       pages={567\ndash 600},
      review={\MR{4739668}},
}

\bib{MR529493}{article}{
      author={Kleiman, Steven~L.},
      author={Altman, Allen~B.},
       title={Bertini theorems for hypersurface sections containing a
  subscheme},
        date={1979},
        ISSN={0092-7872,1532-4125},
     journal={Comm. Algebra},
      volume={7},
      number={8},
       pages={775\ndash 790},
         url={https://doi.org/10.1080/00927877908822375},
      review={\MR{529493}},
}

\bib{MR0382267}{incollection}{
      author={Kleiman, Steven~L.},
       title={Motives},
        date={1972},
   booktitle={Algebraic geometry, {O}slo 1970 ({P}roc. {F}ifth {N}ordic
  {S}ummer {S}chool in {M}ath.)},
   publisher={Wolters-Noordhoff Publishing, Groningen},
       pages={53\ndash 82},
      review={\MR{382267}},
}

\bib{MR2187153}{incollection}{
      author={Kahn, Bruno},
      author={Murre, Jacob~P.},
      author={Pedrini, Claudio},
       title={On the transcendental part of the motive of a surface},
        date={2007},
   booktitle={Algebraic cycles and motives. {V}ol. 2},
      series={London Math. Soc. Lecture Note Ser.},
      volume={344},
   publisher={Cambridge Univ. Press, Cambridge},
       pages={143\ndash 202},
      review={\MR{2187153}},
}

\bib{MR3737321}{article}{
      author={Kahn, Bruno},
      author={Sujatha, Ramdorai},
       title={Birational motives, {II}: {T}riangulated birational motives},
        date={2017},
        ISSN={1073-7928,1687-0247},
     journal={Int. Math. Res. Not. IMRN},
      number={22},
       pages={6778\ndash 6831},
         url={https://doi.org/10.1093/imrn/rnw184},
      review={\MR{3737321}},
}

\bib{MR2738368}{article}{
      author={Kondo, Satoshi},
      author={Yasuda, Seidai},
       title={Product structures in motivic cohomology and higher {C}how
  groups},
        date={2011},
        ISSN={0022-4049},
     journal={J. Pure Appl. Algebra},
      volume={215},
      number={4},
       pages={511\ndash 522},
         url={https://doi.org/10.1016/j.jpaa.2010.06.003},
      review={\MR{2738368}},
}

\bib{MR106225}{book}{
      author={Lang, Serge},
       title={Abelian varieties},
      series={Interscience Tracts in Pure and Applied Mathematics},
   publisher={Interscience Publishers, Inc., New York; Interscience Publishers
  Ltd., London},
        date={1959},
      volume={No. 7},
      review={\MR{106225}},
}

\bib{MR1623774}{book}{
      author={Levine, Marc},
       title={Mixed motives},
      series={Mathematical Surveys and Monographs},
   publisher={American Mathematical Society, Providence, RI},
        date={1998},
      volume={57},
        ISBN={0-8218-0785-4},
         url={https://doi.org/10.1090/surv/057},
      review={\MR{1623774}},
}

\bib{MR3052734}{book}{
      author={Murre, Jacob~P.},
      author={Nagel, Jan},
      author={Peters, Chris A.~M.},
       title={Lectures on the theory of pure motives},
      series={University Lecture Series},
   publisher={American Mathematical Society, Providence, RI},
        date={2013},
      volume={61},
        ISBN={978-0-8218-9434-7},
         url={https://doi.org/10.1090/ulect/061},
      review={\MR{3052734}},
}

\bib{MR0815486}{article}{
      author={Murre, J.~P.},
       title={Abel-{J}acobi equivalence versus incidence equivalence for
  algebraic cycles of codimension two},
        date={1985},
        ISSN={0040-9383},
     journal={Topology},
      volume={24},
      number={3},
       pages={361\ndash 367},
         url={https://doi.org/10.1016/0040-9383(85)90008-4},
      review={\MR{815486}},
}

\bib{MR1061525}{article}{
      author={Murre, J.~P.},
       title={On the motive of an algebraic surface},
        date={1990},
        ISSN={0075-4102,1435-5345},
     journal={J. Reine Angew. Math.},
      volume={409},
       pages={190\ndash 204},
         url={https://doi.org/10.1515/crll.1990.409.190},
      review={\MR{1061525}},
}

\bib{MR1225267}{article}{
      author={Murre, J.~P.},
       title={On a conjectural filtration on the {C}how groups of an algebraic
  variety. {I}. {T}he general conjectures and some examples},
        date={1993},
        ISSN={0019-3577},
     journal={Indag. Math. (N.S.)},
      volume={4},
      number={2},
       pages={177\ndash 188},
         url={http://dx.doi.org/10.1016/0019-3577(93)90038-Z},
      review={\MR{1225267 (94j:14006a)}},
}

\bib{MR2242284}{book}{
      author={Mazza, Carlo},
      author={Voevodsky, Vladimir},
      author={Weibel, Charles},
       title={Lecture notes on motivic cohomology},
      series={Clay Mathematics Monographs},
   publisher={American Mathematical Society},
     address={Providence, RI},
        date={2006},
      volume={2},
        ISBN={978-0-8218-3847-1; 0-8218-3847-4},
      review={\MR{2242284 (2007e:14035)}},
}

\bib{MR1812507}{book}{
      author={Neeman, Amnon},
       title={Triangulated categories},
      series={Annals of Mathematics Studies},
   publisher={Princeton University Press},
     address={Princeton, NJ},
        date={2001},
      volume={148},
        ISBN={0-691-08685-0; 0-691-08686-9},
      review={\MR{MR1812507 (2001k:18010)}},
}

\bib{MR1308405}{article}{
      author={Neeman, Amnon},
       title={The {G}rothendieck duality theorem via {B}ousfield's techniques
  and {B}rown representability},
        date={1996},
        ISSN={0894-0347},
     journal={J. Amer. Math. Soc.},
      volume={9},
      number={1},
       pages={205\ndash 236},
      review={\MR{MR1308405 (96c:18006)}},
}

\bib{MR3034283}{article}{
      author={Pelaez, Pablo},
       title={On the functoriality of the slice filtration},
        date={2013},
        ISSN={1865-2433},
     journal={J. K-Theory},
      volume={11},
      number={1},
       pages={55\ndash 71},
         url={http://dx.doi.org/10.1017/is013001013jkt196},
      review={\MR{3034283}},
}

\bib{MR3614974}{article}{
      author={Pelaez, Pablo},
       title={Mixed motives and motivic birational covers},
        date={2017},
        ISSN={0022-4049},
     journal={J. Pure Appl. Algebra},
      volume={221},
      number={7},
       pages={1699\ndash 1716},
         url={https://doi.org/10.1016/j.jpaa.2016.12.024},
      review={\MR{3614974}},
}

\bib{MR1265529}{incollection}{
      author={Scholl, A.~J.},
       title={Classical motives},
        date={1994},
   booktitle={Motives ({S}eattle, {WA}, 1991)},
      series={Proc. Sympos. Pure Math.},
      volume={55, Part 1},
   publisher={Amer. Math. Soc., Providence, RI},
       pages={163\ndash 187},
         url={https://doi.org/10.1090/pspum/055.1/1265529},
      review={\MR{1265529}},
}

\bib{MR3590347}{article}{
      author={Suslin, Andrei},
       title={Motivic complexes over nonperfect fields},
        date={2017},
        ISSN={2379-1683},
     journal={Ann. K-Theory},
      volume={2},
      number={2},
       pages={277\ndash 302},
         url={https://doi.org/10.2140/akt.2017.2.277},
      review={\MR{3590347}},
}

\bib{MR1764199}{incollection}{
      author={Suslin, Andrei},
      author={Voevodsky, Vladimir},
       title={Relative cycles and {C}how sheaves},
        date={2000},
   booktitle={Cycles, transfers, and motivic homology theories},
      series={Ann. of Math. Stud.},
      volume={143},
   publisher={Princeton Univ. Press},
     address={Princeton, NJ},
       pages={10\ndash 86},
      review={\MR{1764199}},
}

\bib{MR3665001}{article}{
      author={Temkin, Michael},
       title={Tame distillation and desingularization by {$p$}-alterations},
        date={2017},
        ISSN={0003-486X,1939-8980},
     journal={Ann. of Math. (2)},
      volume={186},
      number={1},
       pages={97\ndash 126},
         url={https://doi.org/10.4007/annals.2017.186.1.3},
      review={\MR{3665001}},
}

\bib{MR1764200}{incollection}{
      author={Voevodsky, Vladimir},
       title={Cohomological theory of presheaves with transfers},
        date={2000},
   booktitle={Cycles, transfers, and motivic homology theories},
      series={Ann. of Math. Stud.},
      volume={143},
   publisher={Princeton Univ. Press},
     address={Princeton, NJ},
       pages={87\ndash 137},
      review={\MR{1764200}},
}

\bib{MR1764202}{incollection}{
      author={Voevodsky, Vladimir},
       title={Triangulated categories of motives over a field},
        date={2000},
   booktitle={Cycles, transfers, and motivic homology theories},
      series={Ann. of Math. Stud.},
      volume={143},
   publisher={Princeton Univ. Press},
     address={Princeton, NJ},
       pages={188\ndash 238},
      review={\MR{1764202}},
}

\bib{MR1883180}{article}{
      author={Voevodsky, Vladimir},
       title={Motivic cohomology groups are isomorphic to higher {C}how groups
  in any characteristic},
        date={2002},
        ISSN={1073-7928},
     journal={Int. Math. Res. Not.},
      number={7},
       pages={351\ndash 355},
         url={http://dx.doi.org/10.1155/S107379280210403X},
      review={\MR{1883180 (2003c:14021)}},
}

\bib{MR1977582}{incollection}{
      author={Voevodsky, Vladimir},
       title={Open problems in the motivic stable homotopy theory. {I}},
        date={2002},
   booktitle={Motives, polylogarithms and hodge theory, part i (irvine, ca,
  1998)},
      series={Int. Press Lect. Ser.},
      volume={3},
   publisher={Int. Press, Somerville, MA},
       pages={3\ndash 34},
      review={\MR{MR1977582 (2005e:14030)}},
}

\bib{MR2804268}{article}{
      author={Voevodsky, Vladimir},
       title={Cancellation theorem},
        date={2010},
        ISSN={1431-0635},
     journal={Doc. Math.},
      number={Extra volume: Andrei A. Suslin sixtieth birthday},
       pages={671\ndash 685},
      review={\MR{2804268 (2012d:14035)}},
}

\bib{MR2600283}{article}{
      author={Voevodsky, Vladimir},
       title={Motives over simplicial schemes},
        date={2010},
        ISSN={1865-2433},
     journal={J. K-Theory},
      volume={5},
      number={1},
       pages={1\ndash 38},
         url={http://dx.doi.org/10.1017/is010001030jkt107},
      review={\MR{2600283}},
}

\bib{MR2449178}{book}{
      author={Voisin, Claire},
       title={Hodge theory and complex algebraic geometry. {II}},
     edition={English},
      series={Cambridge Studies in Advanced Mathematics},
   publisher={Cambridge University Press, Cambridge},
        date={2007},
      volume={77},
        ISBN={978-0-521-71802-8},
        note={Translated from the French by Leila Schneps},
      review={\MR{2449178}},
}

\bib{MR0050330}{article}{
      author={Weil, Andr\'{e}},
       title={On {P}icard varieties},
        date={1952},
        ISSN={0002-9327,1080-6377},
     journal={Amer. J. Math.},
      volume={74},
       pages={865\ndash 894},
         url={https://doi.org/10.2307/2372230},
      review={\MR{50330}},
}

\bib{MR1743246}{incollection}{
      author={Weibel, Charles},
       title={Products in higher {C}how groups and motivic cohomology},
        date={1999},
   booktitle={Algebraic {$K$}-theory ({S}eattle, {WA}, 1997)},
      series={Proc. Sympos. Pure Math.},
      volume={67},
   publisher={Amer. Math. Soc., Providence, RI},
       pages={305\ndash 315},
         url={https://doi.org/10.1090/pspum/067/1743246},
      review={\MR{1743246}},
}

\end{biblist}
\end{bibdiv}
\bibliographystyle{abbrv}

\end{document}